\documentclass[11pt]{article}

\usepackage[margin=1in]{geometry}
\usepackage{array, amsmath, amsfonts, amsthm, amssymb, subcaption, titling}
\usepackage{authblk, color, soul, bm, graphicx, empheq, enumitem, bbold}
\usepackage[dvipsnames]{xcolor}
\usepackage[margin = 1.5cm, font = small, labelfont = bf, labelsep = period]{caption}
\usepackage{xspace}
\usepackage[utf8]{inputenc}
\usepackage[T1]{fontenc}
\usepackage[round]{natbib}
\usepackage[ruled]{algorithm2e}
\usepackage{multirow}
\usepackage{mathrsfs}
\usepackage{comment}
\usepackage{eurosym}
\PassOptionsToPackage{hyphens}{url}\usepackage{hyperref}
\usepackage{xurl}
\usepackage{accents}
\usepackage{booktabs, longtable}
\usepackage{mathtools}
\usepackage{multirow}
\usepackage{nicefrac}
\usepackage{cases}
\usepackage{csquotes}
\usepackage{float}
\usepackage{xcolor}
\usepackage{parskip}
\usepackage[normalem]{ulem}

\usepackage{cleveref}

\usepackage[colorinlistoftodos]{todonotes}
\usepackage{comment}

\usepackage[nocomma]{optidef}
\usepackage{etoolbox}
\robustify{\label}

\usepackage{bibunits}

\usepackage{pgfplots, pgfplotstable}
\pgfplotsset{compat=1.18}
\usetikzlibrary {arrows.meta} 

\definecolor{dodgerblue}{rgb}{0.118,0.565,1}
\definecolor{MITRed}{HTML}{750014}

\theoremstyle{plain}

\newtheorem{Prop}{Proposition}

\theoremstyle{definition}

\newtheorem{Rmk}{Remark}
\newtheorem{Ass}{Assumption}
\crefname{Ass}{Assumption}{Assumptions}
\crefname{Prop}{Proposition}{Propositions}
\crefname{lem}{Lemma}{Lemmas}

\newtheorem{lem}{Lemma}
\newcommand{\eg}{\textit{e.g.}}
\newcommand{\ie}{\textit{i.e.}}
\newcommand{\set}[1]{\mathcal{#1}}
\newcommand{\dom}{\mathop{\rm{dom}}}

\newcommand{\ubar}[1]{\underaccent{\bar}{#1}}

\begin{document}
\begin{bibunit}[abbrvnat]
\title{The Value of Storage in Electricity Distribution: The~Role~of~Markets}

\author{
Dirk Lauinger, Deepjyoti Deka, Sungho Shin
}
\date{
  \small
  Massachusetts Institute of Technology \\
  \href{mailto:lauinger@mit.edu}{lauinger@mit.edu},  \href{mailto:deepj87@mit.edu}{deepj87@mit.edu},  \href{mailto:sushin@mit.edu}{sushin@mit.edu}  
}
\maketitle

\begin{abstract}
  Electricity distribution companies deploy battery storage to defer grid upgrades by reducing peak demand. In deregulated jurisdictions, such storage often sits idle because regulatory constraints bar participation in electricity markets. Here, we develop an optimization framework that, to our knowledge, provides the first formal model of market participation constraints within storage investment and operation planning. Applying the framework to a Massachusetts case study, we find that market participation delivers similar savings as peak demand reduction. Under current conditions, market participation does not increase storage investment, but at very low storage costs, could incentivize deployment beyond local distribution needs. This might run contrary to the separation of distribution from generation in deregulated markets. Our framework can mitigate this concern by identifying investment levels appropriate for local distribution needs.
\end{abstract}


\section{Introduction}
\subsection{Background and motivation}

In the 1980s and 1990s, deregulation was en vogue. By breaking up established monopolies, the hope was to increase market efficiency by stimulating competition across various industries, including electricity. In the US, for example, the state of Massachusetts separated electricity generation from distribution and transmission, which were to be kept as regulated monopolies~\citep[ch.~164, sec.~1(d)]{ma1997act}. The rationale for the separation was to increase competition in generation while avoiding any \emph{``cross-subsidization of competitive businesses from regulated businesses and discriminatory policies affecting access to distribution and transmission networks upon which all competitive suppliers depend''}~\citep[p.~12]{joskow2008lessons}. Deregulation occurred at  different rates in different jurisdictions~\citep{borenstein2015restructuring}. To our knowledge, in 13~US states (CT, DE, IL, ME, MD, MA, NH, NJ, NY, OH, PA, RI, TX) and the District of Columbia, investor-owned utilities are restricted to owning either generation assets or transmission and distribution assets. In the European Union, deregulation led to distribution system operators being allowed to own storage only if they cannot contract it from third parties, and prohibited from using storage for market participation~\citep[art.~36(2)(b)]{eu2019storage}. 

In the context of the separation between generation and distribution, storage investment and operation within electricity markets warrant a closer look. Storage has generally been considered generation, but can be considered distribution if it serves grid reliability or defers grid investments such as substation or line upgrades. In the latter case, utilities have been allowed to own and operate storage for distribution needs. Such needs are generally infrequent, especially if they arise from reliability during extreme events. Storage built to address grid reliability thus experiences low utilization, on the order of one discharge cycle per month in some cases~\citep{or2024storage}, and foregoes potential revenue from participating in electricity markets. Allowing for market participation would improve storage economics but risk incentivizing investments that go beyond serving local distribution needs, which would run contrary to the separation of distribution from generation.

Here, we study policy designs that allow for utility-owned storage to participate in wholesale electricity markets while constraining the level of storage investment to local distribution needs. To identify permissible storage investments, we design a model that calculates optimal degrees of investment in grid assets, such as substation and line upgrades, and non-grid resources, such as storage and backup generation. Crucially, this investment model incorporates legal \emph{market participation constraints}, which limit the generation from non-grid resources to serving local distribution needs. Next, we quantify the economic gains from participation in arbitrage and capacity markets, and analyze if the gains lead to storage investments that go beyond meeting local distribution needs.

Our study is timely as US utility spending on distribution grids has increased by 45\% from the year 2018 through 2023 and now stands at over \$50~billion annually~\citep{eia2024gridinvestment}. In addition, some deregulated states, \ie, Maryland (\citeyear[p.~2]{pscmd2019storage_pilot}) and New York~(\citeyear[pp.~12--13]{nypsc2021storage}), have recently allowed market participation, hoping to reduce costs and recognizing that utility-owned distribution-grid storage is unlikely to exert substantial market power. In fact, utility-owned distribution storage accounted for less than~1\% of total US generation capacity as of December~2025, see Appendix~\ref{apx:US_storage}. On the other side of the Atlantic, the European Commission recommends exploring the full flexibility potential of energy storage in distribution grids~\citep[\S~5]{eu2023storage} and entertains proposals to create a market for local services that would help alleviate distribution grid constraints \citep[Art.~29, 34, 40, 41, 44]{acer2025drnc}.
Our work provides a tool that enables jurisdictions to assess whether proposed storage investments meet local distribution needs and to quantify economic gains from market participation.

\subsection{Research questions and contributions}

We contribute to the market participation discussion by answering three research questions:
\begin{enumerate}
    \item How to model \emph{market participation constraints} in storage operation and investment planning?
    \item How do profits from market participation compare to savings from reducing peak demand?
    \item Would market participation generate storage investments that go beyond distribution needs, and if so, how can this be detected?
\end{enumerate}

In answer to Question~1, we model market participation constraints mathematically by limiting the supply from non-grid resources, such as storage, to the shortfall of grid capacity relative to electricity demand. We integrate these constraints in a mixed-integer linear optimization problem that determines distribution grid investment and operating decisions. To our knowledge, this is the first reported approach to formulate such constraints. We address the remaining questions by applying the formulation to a Massachusetts case study. 

Question~2 helps assess if profits from market participation are sufficient to justify contracts with third-party storage or policy changes, considering that regulators have been allowing distribution companies to own storage to defer grid investments. In our case study, the profits amount to between~50\% and~100\% of the savings from reducing peak demand.

Question~3 examines the trade-off between market power concerns and efficient storage utilization. While current utility-owned distribution-grid storage is considered too small to yield much market power, allowing for market participation could incentivize storage investments that exceed local distribution needs and increase market power over time. Our optimization problem can guide investment planning by limiting storage capacity to local grid needs. The trade-off can then be navigated by (\emph{i})~regulators authorizing and incentivizing market participation, while limiting storage investment to local needs; (\emph{ii})~utilities contracting third-party storage; or (\emph{iii})~creating a market for distribution-grid services.

\subsection{Prior work}
In the electricity industry, investment planning is typically done by solving capacity expansion problems~\citep{anderson1972cep}. Existing models~\citep{genx,pypsa} adopt a social planner perspective, in which all resources are jointly controlled by a single agent. Market participation constraints limit coordination between assets and are not included in these models. In distribution grids specifically, engineering studies have focused on reducing peak demand~\citep{martin2019peak, martinez2024distributed} and balancing intermittent generation~\citep{yi2022optimal}, discussing market participation constraints but not modeling them explicitly. A survey among distribution companies confirms that storage investments are, in practice, evaluated primarily based on their ability to defer capital investments \citep[p.~7]{keen2022distribution}, rather than on revenues from market participation.

In terms of operating strategies, storage control relies on heuristic decision rules that limit market participation if storage may be needed for grid support, \eg, by limiting the time or state-of-charge available for market participation~\citep{balducci2019nantucket,or2024storage, lumen2024scaleup}. Such heuristics are overly restrictive when market incentives and grid needs align, enabling storage to meet both objectives without compromising one over the other. Our model quantifies the trade-off between both objectives, which we find to be negligible in our case study.

The market value of storage has previously been analyzed by operations research studies, which show that the joint optimization of participating in multiple markets, \eg, forward, real-time, and reserve markets, yields greater profits than participating in each market separately \citep{lohndorf2023value, lauinger2025storage}. In this vein, \cite{bae2025stacking} explicitly leave the value of peak shaving, which stems from reduced investments in electricity supply, transmission, and distribution, for future work. While the value from reduced investments in electricity supply~\citep{peng2024renewable} and transmission~\citep{qi2015joint} can be captured through capacity markets and locational marginal pricing, respectively, we are not aware of any market mechanism that captures the value from reduced investments in distribution. We address this gap by developing a capacity expansion model with market participation constraints, which allows us to distinguish the market value of storage from its nonmarket value. The distinction is important because it allows us to quantify the benefits of enabling and incentivizing storage to participate in wholesale markets~\citep{wu2023distributed} in addition to providing distribution-grid services.

Similar patterns of market and nonmarket value emerge at different scales. \cite{kaps2025residential} find that German residential customers value the sustainability and autarky benefits from storage about twice as much as the benefits from selling electricity. For distribution grids, we find a similar ratio of~2:1 of the nonmarket value from storage to the market value from arbitrage, which shifts to 1:1 when allowing for capacity market participation. Market profits thus constitute only a fraction of the full value of distributed storage, which motivates the need for optimization models that consider both market and nonmarket value streams. By explicitly incorporating market participation constraints in a capacity expansion problem, we provide such a model.

\subsection{Structure}
The paper unfolds as follows. Section~\ref{sec:model} formulates a mathematical optimization model for storage investment and operation problems with market participation constraints, answering Question~1. Section~\ref{sec:cs} performs numerical experiments to answer Questions~2 and~3. Section~\ref{sec:conclusion} concludes. Appendix~\ref{apx:data} presents the data used in the experiments, Appendix~\ref{sec:opt_formulation} presents a mixed-integer linear reformulation of the optimization problem, and Appendix~\ref{apx:results} lists detailed experimental results.

\paragraph{Notation.} We show vectors in boldface and refer to exogenous electricity demand as electric load.
\section{Problem description and optimization model}\label{sec:model}
Consider a distribution company planning its capital investments. We model the investment decisions in an optimization problem that accounts for market participation constraints.

\subsection{Investment cost}
Over a horizon of $N$ periods, \eg, years, the company decides how much capacity $x_{rn}$ of resource~$ r \in \set{R} = \{\mathrm{b}, \mathrm{g}, \mathrm{s}\}$ to add at the beginning of period~$n$. Investing in resource~($\mathrm{g}$) means expanding the connection to the electricity grid, for example, by adding a distribution line or upgrading a substation. Resources~(b) and~(s) stand for backup generation and storage, respectively. We assume that the cost of investing in any resource~$r \in \mathcal{R}$ in any period~$n$ in a set $\set{N}$, $\vert \set{N} \vert = N$, of planning periods is given by the following extended real-valued function that is
mixed-integer linear representable
\begin{equation}
    c_{rn}(x_{rn}) \coloneq \begin{cases}
        p_{rn} x_{rn} + p^0_{rn} & \mathrm{if}~~\ubar x_r \leq x_{rn} \leq \bar x_r, \\
        0 & \mathrm{if}~~x_{rn} = 0, \\
        \infty & \mathrm{otherwise,}
    \end{cases}
\end{equation}
where $\ubar x_r$ and $\bar x_r$ are lower and upper bounds on admissible investments, and $p_{rn}$ and $p^0_{rn}$ are nonnegative coefficients. The installed capacity of any non-grid resource~$r\in\set{R} \setminus \{\mathrm{g}\}$ at the beginning of period~$n$ is 
\begin{equation}
    \bar x_{rn}(\bm x_r) \coloneq \sum_{i \in \set{I}_r} x^0_{rni} + \sum_{i = \ubar n(n,N_r)}^n x_{ri},
\end{equation}
where $\ubar n(n, N) \coloneq \max\{1, n - N + 1\}$ is an auxiliary function and $\set{I}_r$ is an index set containing preinstalled units with power capacity~$\bm x^0_{rn}$ and lifetime~$N_r$. We introduce a parameter $c \in \set{C} \coloneq \{0,1\}$ to distinguish regular operations ($c = 0$) from contingency operations ($c = 1$). In contingencies, we discount the total installed capacity of grid resources by the largest individual unit, hence
\begin{equation}
    \bar x_{\mathrm{g}nc}(\bm x_\mathrm{g}) \coloneq \sum_{i \in \set{I}_\mathrm{g}} x^0_{\mathrm{g}ni} + \sum_{i = \ubar n(n,N_\mathrm{g})}^n x_{\mathrm{g}i} - c \cdot \max\left\{ 
    \max_{i \in \set{I}_\mathrm{g}}\left\{ x^0_{\mathrm{g}ni} \right\}, \,
    \max_{i \in \{ \ubar n(n, N_\mathrm{g}), \ldots, n \}} x_{\mathrm{g}i}
    \right\},
  \end{equation}
where $\set{I}_\mathrm{g}$ is defined similarly to $\set{I}_\mathrm{b}$ and $\set{I}_\mathrm{s}$. The following proposition characterizes the installed capacity as a function of new investments.

\begin{Prop}\label{prop:xbar}
    The capacity functions $\bar x_{rn}$ and $\bar x_{\mathrm{gnc}}$ are nondecreasing concave piecewise linear for all $r \in \set{R} \setminus \{\mathrm{g}\}$, $n \in \set{N}$, and $c \in \set{C}$.
\end{Prop}

\begin{proof}
    For all $r \in \set{R} \setminus \{\mathrm{g}\}$ and $n \in \set{N}$, the functions $\bar x_{rn}$ are affine and thus concave. For $c = 0$, the function $\bar x_{\mathrm{g}nc}$ is affine as well. For $c = 1$, it is concave piecewise linear because the maximum of affine functions is convex piecewise linear. The functions are nondecreasing because they can be expressed as a minimum of nondecreasing linear functions.
\end{proof}

As expected, the installed capacity is nondecreasing in capacity investments. We account for non-grid resources claiming capacity credits valued at price $\bar{\bm p}$ in the net investment cost function
\begin{equation}
    f(\bm x) \coloneq \sum_{n \in \set{N}} \sum_{r \in \set{R}} c_{rn}(x_{rn}) - \sum_{r \in \set{R} \setminus \{\mathrm{g}\}} \bar p_{rn} \bar x_{rn}(\bm x_r).
\end{equation}
The net investment cost may be increasing or decreasing in capacity investments, depending on the relative magnitude of investment costs to capacity payments.

\subsection{Operating cost}
Each planning period consists of~$J$ operating periods, \eg, days, which themselves consist of~$K$ subperiods, \eg, hours. We distinguish electricity supply from demand, so that all operating decisions are nonnegative. On the supply side, $y^\mathrm{s}_{rnjkc}$ denotes the power supplied by resource $r \in \set{R}$ in subperiod~$k$ of operating period~$j$ of planning period~$n$ under contingency~$c$. Similarly, $y^\mathrm{d}_{rnjkc}$ denotes the power consumed by resource $r \in \set{D} \coloneq \{\mathrm{g}, \mathrm{\ell}, \mathrm{s}\}$. In addition to grid and storage demand, the set $\set{D}$ contains electric load~($\mathrm{\ell}$), \ie, the exogenous electricity demand from customers of the distribution company.

We assume linear operating costs in all periods~$\set{J}$ ($\vert \set{J} \vert = J$) and subperiods~$\set{K}$ ($\vert \set{K} \vert = K$),
\begin{equation}
    g_{nc}(\bm y) \coloneq T_{c} \sum_{j \in \set{J}}\sum_{k \in \set{K}} \left( \sum_{r \in \set{R}} p^\mathrm{s}_{rnjk} y^\mathrm{s}_{rnjkc} - \sum_{r \in \set{D}} p^\mathrm{d}_{rnjk} y^\mathrm{d}_{rnjkc} \right),
\end{equation}
where $T_{c}$ is a probability-weighted time duration and $p^\mathrm{s}_{rnjk}$ and $p^\mathrm{d}_{rnjk}$ are prices for supply and demand, respectively. The operating decisions must obey the following constraints. In every subperiod, supply must equal demand, \ie,
\begin{equation}
    \sum_{r \in \set{R}} y^\mathrm{s}_{rnjkc} = \sum_{r \in \set{D}} y^\mathrm{d}_{rnjkc},~~\forall (n,j,k,c) \in \set{N} \times \set{J} \times \set{K} \times \set{C}.
\end{equation}
Supply and demand must respect capacity limitations, \ie, for all $(n,j,k,c)$ in $\set{N} \times \set{J} \times \set{K} \times \set{C}$,
\begin{subequations}
\label{eq:cap_on_op}
\begin{align}
    & 
    y^\mathrm{s}_{\mathrm{g}njkc} \leq \bar x_{\mathrm{g}nc}(\bm x_\mathrm{g}),~~
    y^\mathrm{s}_{\mathrm{s}njkc} \leq \bar x_{\mathrm{s}n}(\bm x_\mathrm{s}),~~
    y^\mathrm{s}_{\mathrm{b}njkc} \leq \bar x_{\mathrm{b}n}(\bm x_\mathrm{b}),\\
    & 
    y^\mathrm{d}_{\mathrm{g}njkc} \leq \bar x_{\mathrm{g}nc}(\bm x_\mathrm{g}),~~
    y^\mathrm{d}_{\mathrm{s}njkc} \leq \bar x_{\mathrm{s}n}(\bm x_\mathrm{s}),~~
    y^\mathrm{d}_{\ell njkc} \leq \bar y_{\ell njk},
\end{align}
\end{subequations}
where $\bar y_{\ell njk}$ is electric load. If load shedding is not permissible, we set
\begin{equation}
    y^\mathrm{d}_{\ell njkc} = \bar y_{\ell njk}.
\end{equation}
Storage must maintain a state-of-charge between zero and an upper bound given by the product of installed power capacity and storage duration $T^\mathrm{s}$, \ie,
\begin{equation}
    0 \leq y^0_{n} + \Delta t \sum_{l = 1}^k\left( \eta^\mathrm{c} y^\mathrm{d}_{\mathrm{s}njlc} - \frac{y^\mathrm{s}_{\mathrm{s}njlc}}{\eta^\mathrm{d}}\right) \leq T^\mathrm{s} \bar x_{\mathrm{s}n}(\bm x_\mathrm{s}), ~~ \forall (n, j, k, c) \in \set{N} \times \set{J} \times \{0\} \cup \set{K} \times \set{C},
\end{equation}
where $y^0_n$ is an initial state-of-charge target for planning period~$n$, $\Delta t$ is the duration of a subperiod, and $\eta^\mathrm{c}$ and $\eta^\mathrm{d}$ are charging and discharging efficiencies, respectively.

\begin{Rmk}\label{rmk:y0}
    We require the initial state-of-charge to be the same across operating periods and contingency scenarios to limit the use of perfect foresight and ensure smooth transitions from contingency to non-contingency operations. $\hfill \Box$
\end{Rmk}

To mitigate end-of-horizon effects, the terminal state-of-charge in each operating period must equal the initial state-of-charge, \ie,
\begin{equation}
    \sum_{k \in \set{K}} \eta^\mathrm{c} y^\mathrm{d}_{\mathrm{s}njkc} - \frac{y^\mathrm{s}_{\mathrm{s}njkc}}{\eta^\mathrm{d}} = 0, ~~ \forall (n,j,c) \in \set{N} \times \set{J} \times \set{C}.
\end{equation}
To maintain battery warranty, the energy throughput per planning period may not exceed a certain threshold, \eg, 150~discharge cycles per year for the battery in our case study, which gives rise to the constraint
\begin{equation}
    \frac{\Delta t}{\eta^\mathrm{d}} \sum_{j \in \set{J}} \sum_{k \in \set{K}} y^\mathrm{s}_{\mathrm{s}njkc} \leq C^\mathrm{s} T^\mathrm{s}  \bar x_{\mathrm{s}n}(\bm x_\mathrm{s}),
    ~~\forall(n,c) \in \set{N} \times \set{C},
\end{equation}
where $C^\mathrm{s}$ is the admissible number of discharge cycles per planning period. The following proposition characterizes the operational constraints.

\begin{Prop}\label{prop:y}
    For fixed investment decisions, all operational constraints can be represented as a set of linear constraints~$\set{Y}_{nc}(\bm x)$ for all $n \in \set{N}$ and $c \in \set{C}$.
\end{Prop}
\begin{proof}
    The only complicated constraints are the capacity limitations in equations~\eqref{eq:cap_on_op}, which require that operating decisions be smaller than the concave piecewise linear capacity functions $\bar x$. Because concave upper bounds are convex constraints, these conditions are linearly representable.
\end{proof}

Proposition~\ref{prop:y} implies that the operating cost for fixed investment decisions and fixed $(n, c) \in \set{N} \times \set{C}$ are given by the optimal value of the linear program
\begin{equation}\label{pb:g}\tag{OC}
    g^\star_{nc}(\bm x) = \min_{\bm y \in \set{Y}_{nc}(\bm x)} g_{nc}(\bm y),
\end{equation}
which is further characterized by the following proposition.

\begin{Prop}\label{prop:opex}
    The optimal value function~$g^\star_{nc}$ is nonincreasing convex piecewise linear.
\end{Prop}

\begin{proof}
    The set $\set{Y}_{nc}$ is an intersection of halfspaces, each determined by linear inequalities of the form $a^\top \bm y + b \leq \bar x(\bm x)$. The capacity functions~$\bar x$ are nondecreasing concave piecewise linear by Proposition~\ref{prop:xbar}. The claim thus follows from linear programming sensitivity analysis~\citep[Theorem~5.1]{bertsimas1997introduction}.
\end{proof}

As expected, the operational costs are nonincreasing in capacity investments. Minimizing total costs thus requires a trade-off between net investment costs and operating costs, if the former are increasing in capacity investments. 

\subsection{Total cost}
The distribution company aims to minimize investment and operating costs, 
\begin{equation*}\label{pb:tc}\tag{TC}
    \min_x f(\bm x) + \sum_{n \in \set{N}} \sum_{c \in \set{C}} g^\star_{nc}(\bm x).
\end{equation*}
As operating costs are convex and nonincreasing in the capacity investments~$\bm x$, problem~\eqref{pb:tc} models the trade-off between investment costs and operating costs. The computational difficulty in solving~\eqref{pb:tc} stems from the nonconvex noncontinuous extended-real-valued function~$f$, which models the limits on admissible investments and hence the tradeoff between lumpy investments with low per-unit costs (grid investments) and modular investments with high per-unit costs (backup generation and storage). We model this trade-off with $N \vert \set{R} \vert$~binary variables in problem~\eqref{pb:P} in Appendix~\ref{sec:opt_formulation}.

\subsection{Market participation constraints}
In deregulated markets, distribution companies may be allowed to use backup generation and storage to avoid shedding exogenous load, but not for market participation in general. If such restrictions apply, we limit the supply from non-grid resources to the shortfall of grid capacity relative to demand, by imposing the market participation constraints,
\begin{equation}\label{eq:m}
\sum_{r \in \set{R} \setminus \{\mathrm{g}\}} y^\mathrm{s}_{rnjkc} \leq \left[ y^\mathrm{d}_{\ell njkc} - \bar x_{\mathrm{g}nc}(\bm x_\mathrm{g}) \right]^+,~~\forall (n,j,k,c) \in \set{N} \times \set{J} \times \set{K} \times \set{C}.
\end{equation}

For any $(n,c) \in \set{N} \times \set{C}$ and any $\bm x_\mathrm{g} \in \mathbb{R}^N$, the set of operating decisions~$\bm y$ that satisfies these constraints is given by
$
    \set{M}_{nc}(\bm x_\mathrm{g}) = \left\{ \bm y : \eqref{eq:m} \right\}.
$
The following proposition characterizes this set.

\begin{Prop}\label{prop:M}
  For any $(n,c) \in \set{N} \times \set{C}$, the market participation constraints define a nonconvex feasible set~$\set{M}_{nc}(\bm x_\mathrm{g})$ in the operating decisions. The feasible set~$\set{M}_{nc}(\bm x_\mathrm{g})$ is nonincreasing in $\bm x_\mathrm{g}$ in the sense that for any $\bm x_\mathrm{g} \leq \bm x_\mathrm{g}'$, componentwise, $\set{M}_{nc}(\bm x_\mathrm{g}') \subseteq \set{M}_{nc}(\bm x_\mathrm{g})$.
\end{Prop}

\begin{proof}
    The set is nonconvex because constraints~\eqref{eq:m} impose lower bounds on the functions\begin{equation}\label{eq:m_rhs}
        \left[ y^\mathrm{d}_{\ell njkc} - \bar x_{\mathrm{g}nc}(\bm x_\mathrm{g}) \right]^+
        =
        \max\left\{
        0, y^\mathrm{d}_{\ell njkc} -\bar x_{\mathrm{g}nc}(\bm x_\mathrm{g})
        \right\},
    \end{equation}
    which are convex in~$\bm x_\mathrm{g}$ by composition as $\bar x_{\mathrm{g}nc}$ is concave by Proposition~\ref{prop:xbar} and the function $h(x) = \max\{0, - x\}$ is convex and nonincreasing~\citep[eq.~(3.10)]{SB04}. The set inclusion holds because $h$~is nonincreasing and $\bar x_{\mathrm{g}nc}$ nondecreasing.
\end{proof}
The nonconvexities can be modeled by introducing binary variables for each subperiod, which increases the number of binary variables in the overall problem to~$N \vert \set{R} \vert + NJK\vert \set{C} \vert$. Operating decisions that respect the market participation constraints can be found by intersecting the feasible set in problem~\ref{pb:g} with~$\set{M}_{nc}(\bm x_\mathrm{g})$ for each $n \in \set{N}$ and $c \in \set{C}$. For fixed investment decisions and fixed $(n, c) \in \set{N} \times \set{C}$, the operating cost under market participation constraints is thus given by
\begin{equation}\label{pb:coc}\tag{COC}
    \bar g^\star_{nc}(\bm x) = \min_{\bm y \in \set{Y}_{nc}(\bm x) \cap \set{M}_{nc}(\bm x_\mathrm{g})} ~ g_{nc}(\bm y).
\end{equation}

As the sets $\set{M}_{nc}$ are nonincreasing in $\bm x_\mathrm{g}$, imposing the market participation constraints invalidates Proposition~\ref{prop:opex}: We lose convexity and monotonicity in grid investments but retain monotonicity in backup and storage investments, as established in the following proposition.

\begin{Prop}\label{prop:gbar}
    The optimal value function $\bar g^\star_{nc}$ is piecewise linear and nondecreasing in $\bm x_\mathrm{b}$ and $\bm x_\mathrm{s}$. For any $\bm x \in \dom f$, we have $\bar g^\star_{nc}(\bm x) \geq g^\star_{nc}(\bm x)$.
\end{Prop}

\begin{proof}
    For any $(n,c) \in \set{N} \cap \set{C}$, the function is piecewise linear because the case distinction underlying the max-term in constraint~\eqref{eq:m} can be expressed as a series of disjunctive inequalities with auxiliary binary variables, which results in a mixed-binary linear program. For this problem specifically, Theorem~2.1 by \cite{blair1977value} applies because   
    $\bar g^\star_{nc}(0) = 0$. The function is nondecreasing in $\bm x_\mathrm{b}$ and $\bm x_\mathrm{s}$ because the set $\set{Y}_{nc}$ is nondecreasing in~$\bm x$, in the sense that for any $\bm x_\mathrm{g} \leq \bm x'_\mathrm{g}$, componentwise, $\set{Y}_{nc}(\bm x) \subseteq \set{Y}_{nc}(\bm x')$, and the set~$\set{M}_{nc}$ depends on~$\bm x$ only through~$\bm x_\mathrm{g}$. The inequality holds because the feasible set in problem~\eqref{pb:coc} is feasible in problem~\eqref{pb:g}.
\end{proof}

\begin{Rmk}
    Under fixed grid investments and electricity demand decisions, \eg, in the absence of load shedding, constraint~\eqref{eq:m} is linear and can be modeled without binary variables.
\end{Rmk}

In summary, the theoretical analysis reveals that (\emph{i})~operating costs are nonincreasing in backup and storage investments (Proposition~\ref{prop:gbar});
(\emph{ii})~maybe surprisingly, operating costs may be \emph{nondecreasing} in grid investments if market participation is constrained (Proposition~\ref{prop:M}); (\emph{iii})~market participation constraints increase operating costs (Proposition~\ref{prop:gbar}); and (\emph{iv})~computational complexity arises from the market participation constraints and the nonconvex investment cost function, which models the trade-off between modular high-cost resources and lumpy low-cost resources. The market participation constraints and the investment cost function can be modeled using $N J K \vert \set{C} \vert$ and $\vert \set{R} \vert N$ binary variables, respectively. Appendix~\ref{sec:opt_formulation} provides a corresponding mixed-integer linear reformulation.
\section{Numerical experiments}\label{sec:cs}
\subsection{Geographical region}
We study capacity expansion planning on the island of Nantucket, Massachusetts, which suits our purposes because (\emph{i})~it is located in a deregulated state that does not allow storage owned by distribution companies to participate in electricity markets; (\emph{ii})~it features publicly available data about the current energy system and its future evolution; (\emph{iii})~its interconnection to the mainland is limited to two subsea cables, which facilitates analysis; and (\emph{iv})~it is equipped with a battery owned by a distribution company.

The Nantucket Electric Company, a subsidiary of National Grid, serves the entire island and no other territory. In January 2024, the company filed a plan to proactively upgrade its distribution grid with the Massachusetts Department of Public Utilities as required by law~\citep[\S~53]{ma2022act}. The plan describes the current energy system on the island and its projected future evolution through the year~2050 \citep{nationalgrid2024ESMP}. In addition to data from the distribution company, past electricity demand and price data are available from the grid operator serving Nantucket. Specifically, the Independent System Operator New England (ISO-NE) assigns the network node \texttt{LD.CANDLE 13.2} with ID \texttt{16255} in load zone \texttt{4006} (Southeastern Massachusetts) to Nantucket.

In 2019, National Grid installed a 6MW/48MWh battery and a 13MW backup generator on the island to delay investments in an additional subsea cable that would have been needed to ensure supply if one of the two existing cables were to fail. Also in 2019, the Pacific Northwest National Laboratory~(PNNL) released a study about the battery and the generator~\citep{balducci2019nantucket}. 

\subsection{Data}\label{sec:data}
This section describes a few important problem parameters. A detailed list of all parameters with references, mostly pointing to National Grid, the ISO-NE, and PNNL, can be found in Appendix~\ref{apx:data}.

The left panel of Figure~\ref{fig:nantucket} shows the projected evolution of peak load and generation capacity in the absence of new investments from 2025 through 2050 with yearly resolution. In this study, we assume that peak load grows as in National Grid's highest load scenario, \ie, from 62MW in 2025 to 98MW in 2050. In the absence of new investments, the installed capacity falls from 93MW in 2025 to zero by 2047, and N-1 capacity, \ie, total installed capacity minus the largest installed unit, falls from 55MW in 2025 to zero in 2040. We match our case study planning horizon with the 2050 horizon used by National Grid. Over this horizon, it will be necessary to fully renew the existing energy infrastructure and almost double the N-1 capacity.
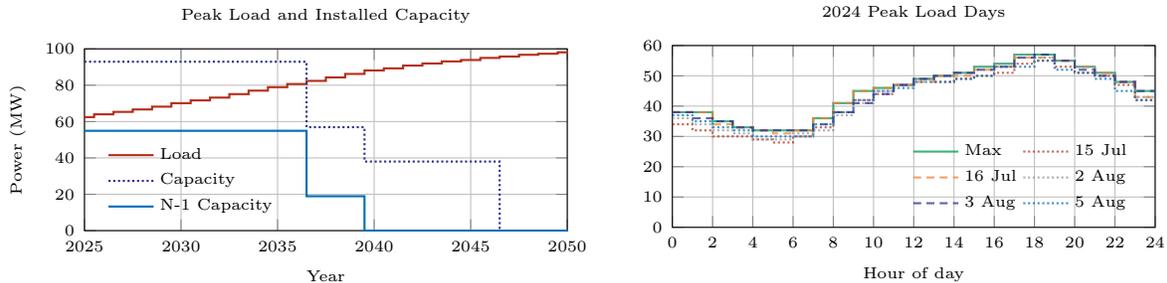
\begin{figure}[!t]
    \centering
    \begin{subfigure}{0.48\textwidth}
    \centering
    \begin{tikzpicture}[font=\tiny]
    \begin{axis}[
        width=8cm, 
        height=4cm, 
        enlarge x limits=false,
        ymin = 0, ymax = 100,
        ytick distance = 20,
        yticklabel style={
        /pgf/number format/fixed,
        /pgf/number format/precision=0
        },
        xticklabel style={
        /pgf/number format/1000 sep={}
    },
        scaled y ticks=false,
        xmin = 2025, xmax = 2050,
        xtick distance=5,
        xlabel={Year},
        ylabel={Power (MW)},
        title={Peak Load and Installed Capacity},
        grid=both, 
        legend style={
        cells={anchor=west},
        draw=none,
        fill=none,
        legend pos= south west,
        font=\tiny
        }, 
    ]
    \addplot+[no marks, solid, const plot, opacity=1, thick, BrickRed] table [x=year, y=load_high, col sep=tab] {data/case_study_parameters.txt};
    \addplot+[no marks, const plot, opacity=1, densely dotted, thick, Blue] table [x=year, y=capa_tot, col sep=tab] {data/case_study_parameters.txt};
    \addplot+[no marks, const plot, opacity=1, solid, thick, NavyBlue] table [x=year, y=capa_N1, col sep=tab] {data/case_study_parameters.txt};
    \legend{
        Load, Capacity, N-1 Capacity
    }
    \end{axis}
    \end{tikzpicture}
    \end{subfigure}
    \begin{subfigure}{0.48\textwidth}
    \centering
    \begin{tikzpicture}[font=\tiny]
    \begin{axis}[
        width=8cm, 
        height=4cm, 
        enlarge x limits=false,
        xmin = 0, xmax = 24,
        ymin = 0, ymax = 60,
        ytick distance = 10,
        yticklabel style={
        /pgf/number format/fixed,
        /pgf/number format/precision=0
        },
        xticklabel style={
        /pgf/number format/1000 sep={}
    },
        scaled y ticks=false,
        xtick distance=2,
        xlabel={Hour of day},
        title={2024 Peak Load Days},
        grid=both, 
        legend style={
        cells={anchor=west},
        draw=none,
        fill=none,
        legend pos= south east,
        font=\tiny,
        legend columns = 2
        }, 
    ]
    \addplot+[no marks, solid, const plot, opacity=0.75, thick, ForestGreen] table [x=Hour, y=Max, col sep=tab] {data/peak_days.txt};
    \addplot+[no marks, densely dotted, const plot, opacity=0.75, thick, BrickRed] table [x=Hour, y=15.07.2024, col sep=tab] {data/peak_days.txt};
    \addplot+[no marks, densely dashed, const plot, opacity=0.75, thick, Orange] table [x=Hour, y=16.07.2024, col sep=tab] {data/peak_days.txt};
    \addplot+[no marks, densely dotted, const plot, opacity=0.75, thick, Gray] table [x=Hour, y=02.08.2024, col sep=tab] {data/peak_days.txt};
    \addplot+[no marks, densely dashed, const plot, opacity=0.75, thick, Blue] table [x=Hour, y=03.08.2024, col sep=tab] {data/peak_days.txt};
    \addplot+[no marks, densely dotted, const plot, opacity=0.75, thick, NavyBlue] table [x=Hour, y=05.08.2024, col sep=tab] {data/peak_days.txt};
    \legend{
        Max, 15 Jul, 16 Jul, 2 Aug, 3 Aug, 5 Aug
    }
    \end{axis}
    \end{tikzpicture}
    \end{subfigure}
    \caption{Electric load and generation capacity in the absence of new investments.}
    \label{fig:nantucket}
\end{figure}

The right panel of Figure~\ref{fig:nantucket} shows load trajectories with hourly resolution for the 5~days with the highest peak load in the year~2024. These days are July 15 and 16, and August 2, 3, and 5. All days follow a similar pattern. Load is lowest at about 30MW in the early morning between 4--6am and highest at 57MW in the early evening between 5--7pm. The large load swings, near-doubling over a day, suggest opportunities for storage to reduce peak load by charging when load is low and discharging when load is high. Given that peak load days tend to follow each other, our modeling choice  to impose a fixed state-of-charge target at the beginning of each day, see Remark~\ref{rmk:y0}, seems appropriate.

Figure~\ref{fig:Ntkt2024} shows the load and electricity price for each hour in the year~2024, as reported by ISO-NE. Consistent with observations from peak load days, load is high on summer evenings and low during the rest of the year, which suggests that storage use for peak load reduction should be concentrated in summer. Electricity price is high on summer evenings and in winter. This suggests that storage operation for peak load reduction is at least partially aligned with storage operation for market participation. In summer, it is best to charge in the morning and discharge in the evening. During the rest of the year, no peak load reduction is needed, and storage could, in principle, be freely used for market participation absent regulatory constraints.

\begin{figure}[t]
    \centering
    \begin{subfigure}{0.45\textwidth}
    \includegraphics[width=0.9\textwidth]{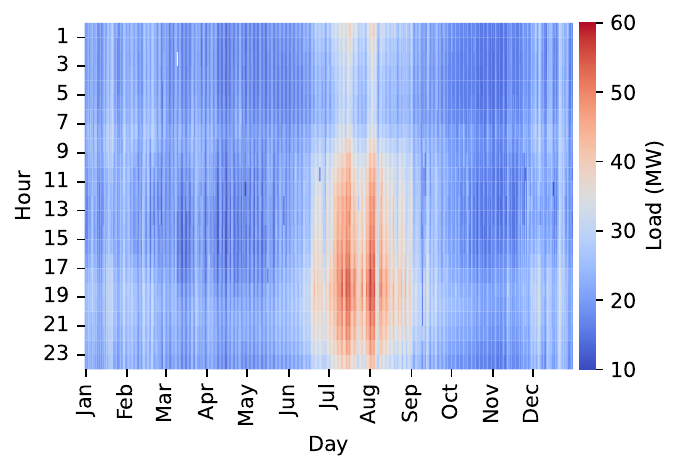}
    \end{subfigure}
    \begin{subfigure}{0.45\textwidth}
    \includegraphics[width=0.83\textwidth]{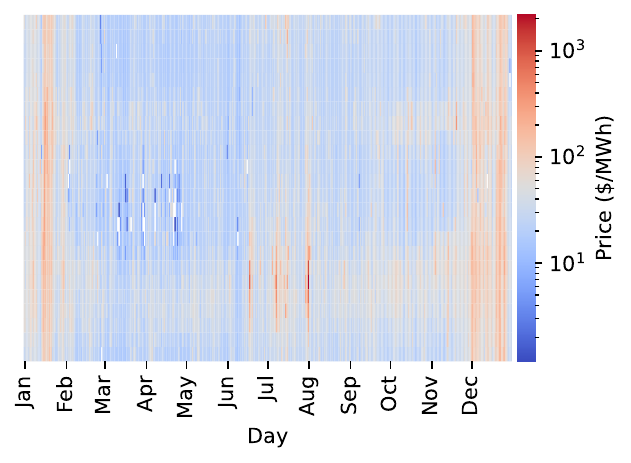}
    \end{subfigure}
    \caption{Nantucket electricity load and price in 2024.}
    \label{fig:Ntkt2024}
\end{figure}

Figure~\ref{fig:scarcity} shows all scarcity events in the year 2024, \ie, times at which the ISO-NE did not have sufficient generation capacity to reliably meet demand. In these times, the ISO-NE calls on generation providers that had been awarded supply obligations in a forward capacity market. Supply obligation holders are remunerated at the capacity price shown in Figure~\ref{fig:fcm_price}, which varies on a yearly basis. The year~2024 had two scarcity events, both on summer evenings. This suggests that storage operation for peak load reduction is aligned with capacity market participation.
\begin{figure}[!t]
    \centering
    \begin{subfigure}{0.48\textwidth}
    \centering
    \begin{tikzpicture}[font=\tiny]
    \begin{axis}[
        width=8cm, 
        height=4cm, 
        enlarge x limits=false,
        ymin = 0, ymax = 1,
        ytick distance = 0.2,
        yticklabel style={
        /pgf/number format/fixed,
        /pgf/number format/precision=1
        },
        xticklabel style={
        /pgf/number format/1000 sep={}
    },
        scaled y ticks=false,
        xmin = 0, xmax = 24,
        xtick distance=2,
        xlabel={Hour of day (-)},
        ylabel={Ratio (-)},
        title={18 June 2024},
        grid=both, 
        legend style={
        cells={anchor=west},
        legend pos= north west,
        font=\tiny
        }, 
    ]
    \addplot [
        draw=none,
        fill=gray,
        fill opacity=0.3,
        on layer=axis background,
        legend image code/.code={%
        \draw[fill=gray, fill opacity=0.3, draw=none] (0cm,-0.1cm) rectangle (0.6cm,0.2cm);%
        }
    ] coordinates {
        (17.83,0)
        (18.33,0)
        (18.33,1)
        (17.83,1)
    };
    \addlegendentry{Scarcity}
    \addplot+[no marks, solid, const plot, opacity=1, very thick, NavyBlue] table [x=Hour, y=Price, col sep=tab] {data/scarcity_Jun18.txt};
    \addlegendentry{Price}
    \addplot+[no marks, densely dotted, const plot, opacity=1, very thick, Orange] table [x=Hour, y=Load_24, col sep=tab] {data/scarcity_Jun18.txt};
    \addlegendentry{Actual load}
    \addplot+[no marks, densely dashed, const plot, opacity=1, very thick, BrickRed] table [x=Hour, y=Peak_load, col sep=tab] {data/scarcity_Jun18.txt};
    \addlegendentry{Peak load}
    \end{axis}
    \end{tikzpicture}
    \end{subfigure}
    \begin{subfigure}{0.48\textwidth}
    \centering
    \begin{tikzpicture}[font=\tiny]
    \begin{axis}[
        width=8cm, 
        height=4cm, 
        enlarge x limits=false,
        ymin = 0, ymax = 1,
        ytick distance = 0.2,
        yticklabel style={
        /pgf/number format/fixed,
        /pgf/number format/precision=1
        },
        xticklabel style={
        /pgf/number format/1000 sep={}
    },
        scaled y ticks=false,
        xmin = 0, xmax = 24,
        xtick distance=2,
        xlabel={Hour of day (-)},
        title={1 August 2024},
        grid=both, 
        legend style={
        cells={anchor=west},
        draw=none,
        fill=none,
        legend pos= south west,
        font=\tiny
        }, 
    ]
    \addplot [
        draw=none,
        fill=gray,
        fill opacity=0.3,
        on layer=axis background,
        legend image code/.code={%
        \draw[fill=gray, fill opacity=0.3, draw=none] (0cm,-0.1cm) rectangle (0.6cm,0.2cm);%
        }
    ] coordinates {
        (16.916,0)
        (17.083,0)
        (17.083,1)
        (16.916,1)
    };
        \addplot [
        draw=none,
        fill=gray,
        fill opacity=0.3,
        on layer=axis background,
        legend image code/.code={%
        \draw[fill=gray, fill opacity=0.3, draw=none] (0cm,-0.1cm) rectangle (0.6cm,0.2cm);%
        }
    ] coordinates {
        (17.750,0)
        (19.417,0)
        (19.417,1)
        (17.750,1)
    };
    \addplot+[no marks, solid, const plot, opacity=1, very thick, NavyBlue] table [x=Hour, y=Price, col sep=tab] {data/scarcity_Aug01.txt};
    \addplot+[no marks, densely dotted, const plot, opacity=1, very thick, Orange] table [x=Hour, y=Load_24, col sep=tab] {data/scarcity_Aug01.txt};
    \addplot+[no marks, densely dashed, const plot, opacity=1, very thick, BrickRed] table [x=Hour, y=Peak_load, col sep=tab] {data/scarcity_Aug01.txt};
    \end{axis}
    \end{tikzpicture}
    \end{subfigure}
    \caption{Scarcity events in the ISO-NE in 2024; locational marginal prices and load are normalized by their peak values in the year (\$2173.15/MWh and 57MW).}
    \label{fig:scarcity}
\end{figure}
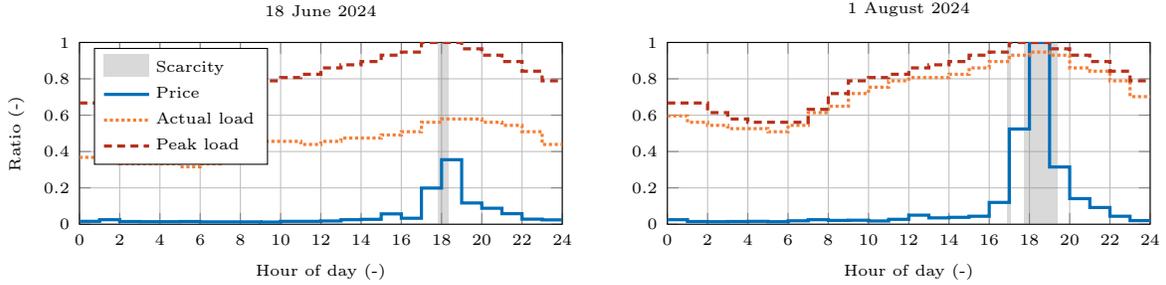
\begin{figure}[!t]
    \centering
    \begin{tikzpicture}[font=\tiny]
    \begin{axis}[
        width=17cm, 
        height=4cm, 
        enlarge x limits=false,
        ymin = 0, ymax = 12,
        ytick distance = 2,
        yticklabel style={
        /pgf/number format/fixed,
        /pgf/number format/precision=0
        },
        xticklabel style={
        /pgf/number format/1000 sep={},
        rotate = 90
    },
        scaled y ticks=false,
        xmin = 2010, xmax = 2050,
        xtick distance=1,
        xlabel={Year},
        title={Capacity price (\$/kW-month)},
        grid=both, 
        legend style={
        cells={anchor=west},
        legend pos= north east,
        font=\tiny
        }, 
    ]
    \addplot [
        draw=none,
        fill=gray,
        fill opacity=0.2,
        on layer=axis background,
        legend image code/.code={%
        \draw[fill=gray, fill opacity=0.2, draw=none] (0cm,-0.1cm) rectangle (0.6cm,0.2cm);%
        }
    ] coordinates {
        (2028.5,0)
        (2050,0)
        (2050,12)
        (2028.5,12)
    };
    \addlegendentry{Forecast}
    \addplot+[no marks, solid, const plot, opacity=1, very thick, NavyBlue] table [x=year, y=price, col sep=tab] {data/fcm.txt};
    \addlegendentry{Price}
    \end{axis}
    \end{tikzpicture}
    \caption{Forward price for existing capacity in Southeastern Massachusetts.}
    \label{fig:fcm_price}
\end{figure}

\subsection{Assumptions}
We make the following simplifying assumptions.
\begin{enumerate}
    \item \textbf{Constant resource parameters:} The capacity and efficiency of all resources stay the same over their lifetimes. In practice, the resources will experience degradation. While the manufacturer of the Nantucket battery guaranteed the nominal capacity over a 20-year lifespan conditional on regular maintenance and respect of a yearly maximum number of discharge cycles, the roundtrip efficiency will degrade, which we account for by using a lifetime average efficiency, similar to Balducci et al. (\citeyear[p.~3.3]{balducci2019nantucket}). We assume that market participation does not cause any additional degradation, including capacity degradation, as long as the yearly cycle limit is respected. In practice, \cite{or2024storage} report that market participation requires one maintenance day per month. Here, we neglect any downtime requirements. These assumptions can be relaxed by adapting problem parameters. For example, battery degradation could be considered more explicitly via positive charging and discharging costs~$\bm p^\mathrm{d}_{\mathrm{s}}$ and~$\bm p^\mathrm{s}_{\mathrm{s}}$.
    \item \textbf{Simplified battery dynamics:} The maximum charge and discharge power and charging and discharging efficiencies are independent of operational parameters. In reality, they would depend on the state-of-charge and temperature, among others. Such dependencies could be partly addressed by maintaining the state-of-charge within a limited range. In addition, we ignore the complementarity constraint between charging and discharging, and check \emph{ex-post} for violations. While there were none in our experiments, they could happen, in principle, when electricity prices are negative.
    \item \textbf{Limited foresight}: Prices, demand, and the occurrence of contingencies are used one day ahead of time for operating decisions. In addition, prices and demand are used one year ahead of time for the allocation of the annual battery discharge budget of 150~cycles and to decide on the fixed state-of-charge target for the beginning of each day. Finally, prices and demand are used up to 26~years ahead of time, the length of our planning horizon, for investment decisions. These assumptions can be relaxed by considering multiple scenarios, at the price of increased computational complexity. On an operational level, this may not be worthwhile because a near-optimal allocation of discharge cycles may be determined~\emph{ex-ante}. In fact, the data analysis in Section~\ref{sec:data} reveals that summer is by far the most promising time of year for storage operations, suggesting that most of the discharge budget should be spent then. The analysis also shows that peak load days tend to be similar and follow each other, suggesting that the state-of-charge target can be based on a few select peak load days.
\end{enumerate}

\subsection{Numerical implementation}
All numerical experiments are conducted on AMD EPYC 9474F CPUs with 48~cores, a 3.6GHz base clock, and 376GB of RAM. Simulations are implemented in Julia 1.11.2
using JuMP 1.23.5 with Gurobi 12.0.2. All code and data are available at \url{https://github.com/mit-shin-group/storage-value}.

\subsection{Experimental setup and results}\label{sec:experiments}
We solve the investment problem~\eqref{pb:tc} formulated as the mixed-integer linear problem~\eqref{pb:P} in Appendix~\ref{sec:opt_formulation} for the case study over a 26~year planning horizon with hourly resolution. We perform nine experiments, varying market participation constraints, available resources for investments, storage costs, storage cycle limits, and forward capacity prices. Table~\ref{tab:cs_results} in Appendix~\ref{apx:results} reports total cost, solve time, mixed-integer programming gaps, operating and capital costs, revenue from capacity payments, investment and operating decisions, yearly discharge cycles, and the minimum supply ratio during scarcity events, defined as the minimum power generation during the event divided by the installed generation capacity of backup generation or storage. Figure~\ref{fig:investment} in Appendix~\ref{apx:results} reports the investment decisions into each resource in each planning period.

Experiment~1 allows for investment in grid expansion only, bars market participation, and results in a total cost of \$678.9~million, which serves as a benchmark for comparison with all other experiments. Experiment~2 allows for grid and storage investments, which reduces grid investments from 160MW to 120MW in exchange for a 19.3MW storage investment. Total costs are reduced by~4.7\%. Market participation is constrained in this experiment, modeling the status quo in Massachusetts, which leads to storage experiencing 7.0 discharge cycles per year on average during contingencies and none during normal operations. Effectively, storage is thus kept sitting idle at a high state-of-charge for the vast majority of the time and discharged only when needed to meet peak electricity demand.

Experiment~4 allows for arbitrage on wholesale markets and participation in the ISO-NE capacity market, which decreases total costs by an additional 4.5\% and increases storage utilization to 150 discharge cycles per year during both normal operations and contingencies, but does not trigger any additional storage investment. As backup generation and storage operate at maximum capacity during scarcity events, neither would pay any penalty for failing to meet capacity supply obligations. This may seem surprising, given that the problem formulation does include explicit information about scarcity events, but can be explained by the high wholesale prices during scarcity events. Under the~2024 load pattern an 8~hour storage duration is thus sufficient to reliably meet capacity supply obligations, which is consistent with ISO-NE assigning an effective load carrying capacity of~1 to storage with a duration of over 2~hours. Since future years will be different from~2024, we perform Experiment~3,  which allows for arbitrage but not capacity market participation. In this case, total costs decrease by~2.0\%, otherwise by~4.5\%.

Experiment~5 evenly splits the 150 yearly cycle budget across each day, which results in a cost reduction of only~2.9\% because almost twice as much storage is needed to reduce peak demand, compared to a~6.7\% reduction in Experiment~3, which optimized the discharge cycle allocation.

Experiments~6 and~7 allow for backup generation investments and arbitrage without and with capacity markets, respectively. Compared to Experiments~3 and~4, backup generation investments increase from~0MW to~22.2MW, while storage investments decrease from 19.3MW to 11.6MW, which reflects the initial investment cost advantage of backup generation over storage, see Figure~\ref{fig:cs_params}. Total costs decrease by an additional 0.5--2.5\%.

Experiments~8 and~9 model deployment if storage were virtually free, \ie, have a capital cost of \$1/kWh. Without market participation, deployment increases from 19.3MW in Experiments~2--4 to 21.3MW. With arbitrage, storage deployment increases to 391.9MW, far greater than the needs for demand reduction. Storage deployment does not increase indefinitely because the gains from market participation do not outweigh the costs of grid build-out.

In the three experiments~(1, 2, 8) without market participation, contingency operations are on average~$0.35$\% cheaper than base-case operations. As the only difference between contingency and base-case operations is a reduced access to the electricity grid, this finding is predicted by Proposition~\ref{prop:gbar}. With market participation, base-case operations are less costly than contingency operations, as established theoretically in  Proposition~\ref{prop:opex}. In Experiments~3--5, contingency operations are no more than~0.1\% more expensive than base-case operations. In Experiment~9, which features virtually free storage, contingency operations are~50\% more expensive than base-case operations. 

Figure~\ref{fig:costs} visualizes the cost savings from storage with and without market participation.
Figure~\ref{fig:hm_supply} shows the supply decisions for each resource type for the year~2025 for Experiments~2 and~3. The figure confirms that, without market participation, backup generation and storage is only used for contingency operations. With market participation, contingency and base-case operations are similar, which explains why their costs are similar. 

\begin{figure}[!t]
    \centering
    \begin{tikzpicture}[font=\tiny]
\begin{axis}[
    width = \textwidth,
    height = 4.5cm,
    bar width = 2cm,
    ybar, ymin = 0, ymax = 5.5, ytick = {0, 1,...,5},
    ylabel={Savings (\%)},
    symbolic x coords={A, B, C},
    xtick={A,B,C},
    xticklabels={
    {Deferred investment},
    {Capacity market},
    {Wholesale arbitrage}
    },
    xticklabel style={align=center}, 
    nodes near coords,
    every node near coord/.append style={
    /pgf/number format/fixed,
    /pgf/number format/precision=1,
    /pgf/number format/zerofill,
    yshift=1pt,
    }, 
    nodes near coords align={vertical},
    xtick pos=bottom,
    ytick pos=left,
    enlarge y limits=false,
    ymin=0,
    ymajorgrids=true,
    legend cell align={left},
]
    \addplot+[ybar, bar shift=0pt] coordinates {(A,4.7)};
    \addlegendentry{Grid value}
    
    \addplot+[ybar, bar shift=0pt] coordinates {(B,2.5) (C,2.0)};
    \addlegendentry{Market value}
\end{axis} 
\end{tikzpicture}
    \caption{Savings from storage investment in distribution grids relative to total investment and operating costs in the absence of storage or backup investment.}
    \label{fig:costs}
\end{figure}

\begin{figure}[!t]
    \centering
    \includegraphics[width=1\linewidth]{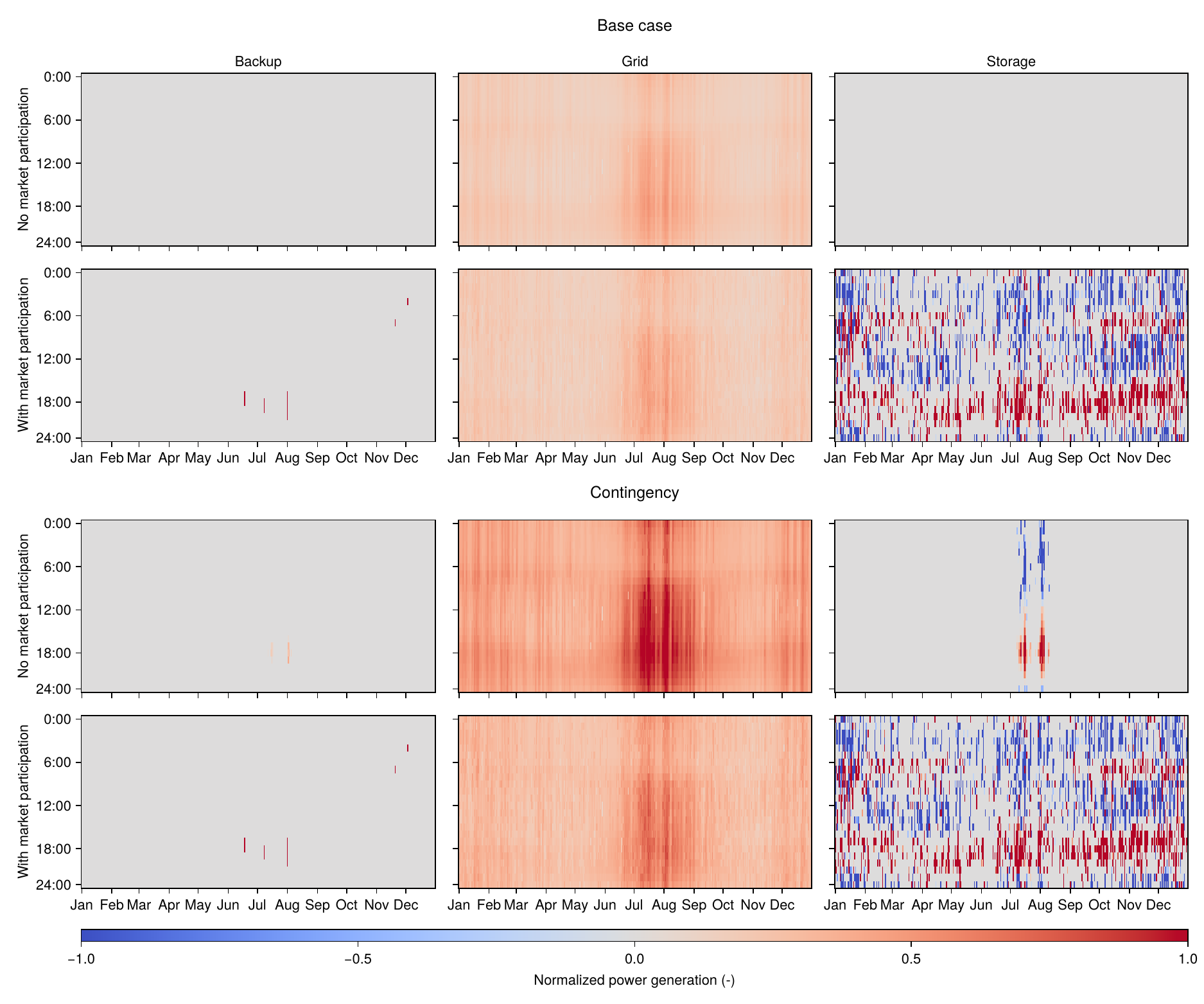}
    \caption{Supply decisions for the year~2025, normalized by installed capacity, for Experiments~2 and~3.}
    \label{fig:hm_supply}
\end{figure}

In summary, the numerical experiments reveal that arbitrage and capacity market participation each increase the cost savings from allowing for storage investment by about~50\%. For Nantucket, this translates into savings of up to \$30~million over 26~years, which could be distributed to rate payers. Under current prices, market participation generates no additional storage investment. Specifically, storage investment is the same in Experiments~2, 3, and 4, and in Experiments~5 and 6. In addition, the trade-off between operating storage for grid support versus market participation is negligible. In fact, with market participation, contingency operations were no more than~$0.1\%$ more expensive than base-case operations in Experiments~3--5, which may be attributed to the correlation between distribution-level electricity demand and Massachusetts-wide wholesale market prices. Finally, under very low costs, market participation does generate additional storage investment, which may be undesirable in deregulated markets because it runs contrary to the idea of separating generation from distribution and transmission.     In this case, the market participation constraint~\eqref{eq:m} can be used to find the appropriate investment level for distribution needs. As the constraint is nonconvex, it increases computational complexity. If that increase is judged excessive, the analytical model in Appendix~\ref{sec:max_peak_shaving} can be used to determine an upper bound on appropriate investment levels.
\section{Discussion and conclusion}\label{sec:conclusion}

We addressed three questions about the value of storage in distribution grids:
\begin{enumerate}
    \item How to model market participation constraints in storage operation and investment planning;
    \item How do profits from market participation compare to savings from reducing peak demand;
    \item Would market participation generate storage investments that go beyond distribution needs, and if so, how can this be detected? 
\end{enumerate}

Equation~\eqref{eq:m} models market participation constraints by limiting the supply from non-grid resources, such as storage, to the shortfall of grid capacity relative to electricity demand. We integrate these constraints into an optimization problem that determines investment decisions and mimics integrated resource planning by distribution grid companies. As the constraints are nonconvex, they increase the computational complexity of the problem, raising questions about tractability. In a Massachusetts case study, we find that problems with a 26~year horizon and hourly resolution can be solved within several hours on current compute servers.  

The case study further reveals that arbitrage and capacity market participation each generate about~50\% of the capital cost savings from reducing or deferring grid investments. We determine storage investment levels appropriate for serving local distribution needs by solving the planning problems with the market participation constraints. We find that under current technology costs, market participation does \emph{not} generate any storage investment that go beyond distribution needs. 

Battery storage can thus be used more efficiently by providing distribution grid services and participating in electricity markets. These efficiency gains could be realized by (\emph{i})~allowing and incentivizing distribution companies to participate in electricity markets, (\emph{ii})~incentivizing distribution companies to contract third-party storage that also participates in electricity markets, (\emph{iii})~creating a market for distribution grid services, potentially relying on nodal pricing in distribution networks~\citep{sotkiewicz2006nodal}, or a combination of these approaches. We now discuss how our findings inform each of these options.

Regarding market participation of storage owned by distribution companies, our model can be used to estimate the value of such a policy change. In the past, the prospect of reduced or deferred grid investments was large enough to cast aside fears about market distortion, and regulators allowed storage investment for distribution needs. Market participation promising similar savings may prompt regulators to reconsider existing policies that restrict market participation for storage assets owned by distribution companies. There is a risk that allowing for market participation would incentivize distribution companies to invest in storage solely for serving the market. Under current market conditions, we find that such investments are not profitable. Even if they were profitable, our model could be used to audit proposed investments and limit them to levels appropriate for addressing distribution needs. Jurisdictions that have already allowed market participation, \eg, New York and Maryland, may find our model useful to audit proposed storage investments.

Regarding distribution companies contracting third-party storage, our model can be used to inform contract terms, \eg, storage quantities, prices, and commitment durations. The substantial savings from deferring grid investments shown in our case study suggest that such contracts could be beneficial to both distribution companies and storage developers, which may make previously unprofitable storage projects viable. 

A market for distribution grid services could in theory be organized similarly to existing capacity markets on the transmission level. However, each distribution grid would likely have far fewer sellers than a capacity market, which may create concerns about liquidity and market power. Our model could be used to inform the terms of the capacity auctions in such a market, \eg, price caps and product durations.

\textbf{Acknowledgements:} We thank Ruby Aidun and Dr.~John Parsons at the MIT Center for Energy and Environmental Policy Research for guidance on US distribution grid policy, Prof.~Paul Joskow at MIT Economics for guidance on electricity markets, the MIT Future Energy Systems Center for funding, the MIT Office of Research Computing and Data for providing high performance computing resources, and Finn Ye, Terron Hill, Colette Lamontagne from National Grid and Dr.~Iason Iraklis Avramidis from National Grid Partners for guidance on distribution grid best practices.

\addcontentsline{toc}{section}{References (Main Text)}
\small
\linespread{1}\selectfont
\putbib[_mybib]
\end{bibunit}

\newpage

\normalsize
\appendix
\begin{bibunit}[abbrvnat]
\renewcommand{\theequation}{\Alph{section}\arabic{equation}}
\renewcommand{\thefigure}{\Alph{section}\arabic{figure}} 
\renewcommand{\thetable}{\Alph{section}\arabic{table}}
\renewcommand{\theProp}{\Alph{section}\arabic{Prop}}
\setcounter{equation}{0}
\setcounter{Prop}{0}
\setcounter{table}{0}
\setcounter{figure}{0}
\section{Data}\label{apx:data}
\subsection{Installed US electricity storage}\label{apx:US_storage}
Based on the preliminary monthly electric generator inventory by the US Energy Information Administration for December~2025, available at \url{https://www.eia.gov/electricity/data/eia860m/}, the total operating nameplate capacity of the US electricity sector was 1349.6GW, out of which 42.6GW was batteries, out of which 4.3GW were owned by utilities. Utility-owned battery storage thus accounts for about 0.3\% of total US generation capacity.

\subsection{Battery costs and case study parameters}
\begin{figure}[h]
    \centering
    \includegraphics[width=0.5\linewidth]{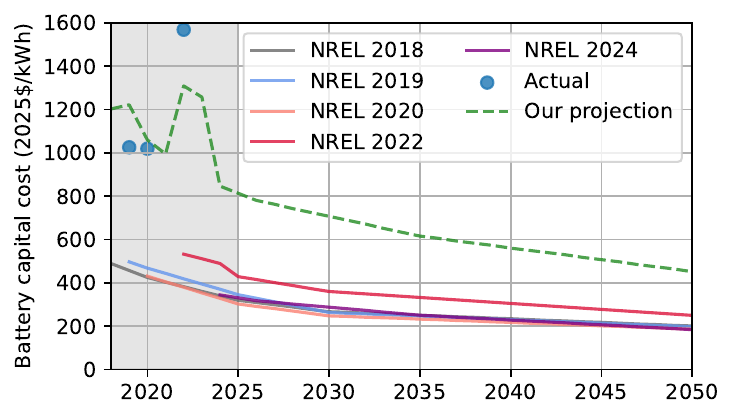}
    \caption{Projected and actual battery costs per kWh of energy capacity (in 2025 USD).
      Solid lines are NREL projections, blue dots are actual distribution grid battery projects in Provincetown and Nantucket, both MA, and Ponoma, NY. The dashed green line is our projection based on the factor by which actual costs exceed NREL projections. The mean storage investment cost is \$604/kWh, \ie, \$4832/kW for 8h~duration.}
    \label{fig:battery_capital_cost}
\end{figure}

\begin{figure}[htbp]
    \centering
    \includegraphics[width=\textwidth]{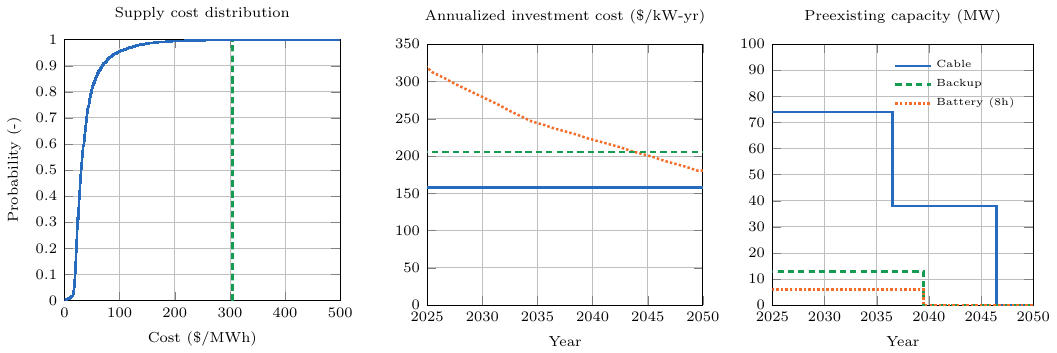}
    \caption{Case study generation cost, investment cost, and preexisting capacity.}
    \label{fig:cs_params}
\end{figure}

\subsection{Maximum peak shaving potential}\label{sec:max_peak_shaving}
It is useful to assess upper bounds on yearly storage investment and total installed storage capacity to determine the constant~$\bar x_\mathrm{s}$ in the investment planning problem~\eqref{pb:tc}. In this section, we will show how these bounds can be estimated via a simplified linear program.

Consider a storage device with infinite power and energy capacity and an infinite grid connection used to fully flatten electric load. The flattened load in any operating period~$j \in \set{J}$ of planning period $n \in \set{N}$ is given by
\begin{equation}\label{pb:FL}\tag{FL}
    y^\star_{\ell n j}(\eta^\mathrm{c} \eta^\mathrm{d}) = \min \, \max_{k \in \set{K}} \, \bar y_{\ell n j k} - y^\mathrm{s}_{\mathrm{s}njk} + y^\mathrm{d}_{\mathrm{s}njk}~~\mathrm{s.t.}~~\sum_{k \in \set{K}} \eta^\mathrm{c} \eta^\mathrm{d} y^\mathrm{d}_{\mathrm{s}njk} - y^\mathrm{s}_{\mathrm{s}njk} \geq 0, \, \bm y^\mathrm{d}_{\mathrm{s}nj}, \bm y^\mathrm{s}_{\mathrm{s}nj} \geq 0,
\end{equation}
which can be formulated as a linear program that only depends on roundtrip efficiency and load.

\begin{Prop}\label{prop:potential}
    The flattened load function is convex nonincreasing and ranges from $\max_{k \in \set{K}}  \bar y_{\ell njk}$ for $\eta^\mathrm{c} \eta^\mathrm{d} = 0$ to $\frac{1}{K}\sum_{k \in \set{K}} \bar y_{\ell njk}$ for $\eta^\mathrm{c} \eta^\mathrm{d} = 1$.
\end{Prop}

\begin{proof}
    Problem~\eqref{pb:FL} admits the dual formulation
    \begin{equation*}
        \max_{\lambda, \, \bm \mu} \,
        \sum_{k \in \set{K}} \bar y_{\ell njk} \mu_k
        ~~ \mathrm{s.t.} ~~
        \sum_{k \in \set{K}} \mu_k = 1, \,
        \eta^\mathrm{c} \eta^\mathrm{d} \lambda \leq \mu_k \leq \lambda ~~ \forall k \in \set{K}, \, \lambda \geq 0.
    \end{equation*}
    Since the objective function is linear and the feasible set a polyhedron, there exists an optimal solution at a vertex of the polyhedron. Let $\bar{\bm y}$ be ordered such that $\bar y_{\ell nj1} \geq \ldots \geq \bar y_{\ell njK}$. Then, there exists an optimal solution of the form $\mu_{1,\ldots,m} = \lambda$, $\mu_{m+1,\ldots,K} = \eta^\mathrm{c} \eta^\mathrm{d} \lambda$, and $\lambda = \frac{1}{m + (K-m)\eta^\mathrm{c} \eta^\mathrm{d}}$, \ie, at a vertex of the box constraints that satisfies the sum constraint, for some $m \in [1, K] \cap \mathbb{Z}$. Let
    \begin{equation*}
        \varphi_m(\eta^\mathrm{c} \eta^\mathrm{d})
        \coloneq
        \frac{\sum_{i=1}^m \bar y_{\ell nji} + \eta^\mathrm{c} \eta^\mathrm{d} \sum_{i=m+1}^K \bar y_{\ell nji}}
        {
        m + (K-m) \eta^\mathrm{c} \eta^\mathrm{d}
        }.
    \end{equation*}
    The optimal value of the dual problem as a function of~$\eta^\mathrm{c} \eta^\mathrm{d}$ is thus
    \begin{equation*}
        \varphi(\eta^\mathrm{c} \eta^\mathrm{d})
        \coloneq
        \max_{m} \,
        \varphi_m(\eta^\mathrm{c} \eta^\mathrm{d})
        ~~ \mathrm{s.t.} ~~
        m \in [1, K] \cap \mathbb{Z}.
    \end{equation*}
    Evaluating first derivatives, we find
    \begin{equation*}
        \varphi_m(\eta^\mathrm{c} \eta^\mathrm{d})'
        =
        \frac{m \sum_{i = m+1}^K \bar y_{\ell nji} - (K-m) \sum_{i=1}^m \bar y_{\ell nji}}{(m + (K-m)\eta^\mathrm{c} \eta^\mathrm{d})^2},
    \end{equation*}
    which is nonpositive because
    \begin{equation*}
        \frac{\sum_{i = m+1}^K \bar y_{\ell nji}}{K-m}
        \leq
        \frac{\sum_{i = 1}^m \bar y_{\ell nji}}{m},
    \end{equation*}
    as $\bar y_{\ell nj1} \geq \ldots \geq \bar y_{\ell njK}$, and thus increasing in~$\eta^\mathrm{c} \eta^\mathrm{d}$. The functions~$\varphi_m$ are therefore convex nonincreasing and so is their pointwise maximum~$\varphi$. In addition, $\varphi(0) = \max_{k \in \set{K}} \bar y_{\ell njk}$ and $\varphi(1) = \frac{1}{K}\sum_{k \in \set{K}} \bar y_{\ell njk}$. The dual problem thus admits finite optimal values and strong linear programming duality holds. Therefore, the optimal value functions of the primal and dual problems coincide. 
\end{proof}

For any given roundtrip efficiency and load profile, we can compute the power and energy capacity required to fully flatten load based on the optimizers to problem~\eqref{pb:FL} as
\begin{enumerate}
    \item $\max_{k \in \set{K}} \left\{ y^\mathrm{s\star}_{\mathrm{s}njk}, \, y^\mathrm{d\star}_{\mathrm{s}njk} \right\} $ for power capacity and
    \item $\max_{k \in \set{K}} \Delta t \sum_{l = 1}^k \left( \eta^\mathrm{c} y^\mathrm{d\star}_{\mathrm{s}njk} - \frac{y^\mathrm{s\star}_{\mathrm{s}njk}}{\eta^\mathrm{d}} \right) 
    - \min_{k \in \set{K}} \sum_{l = 1}^k \left( \eta^\mathrm{c} y^\mathrm{d\star}_{\mathrm{s}njk} - \frac{y^\mathrm{s\star}_{\mathrm{s}njk}}{\eta^\mathrm{d}} \right)$ for energy capacity.
\end{enumerate}
Figure~\ref{fig:potential} shows the peak-shaving potential for the load profile in Figure~\ref{fig:nantucket}. For typical battery roundtrip efficiencies of~85\% to~95\%, the flattened load will be about 1\% higher than the average load and require a power capacity of~30\% of the average load and a storage duration of~$6.5$~hours.
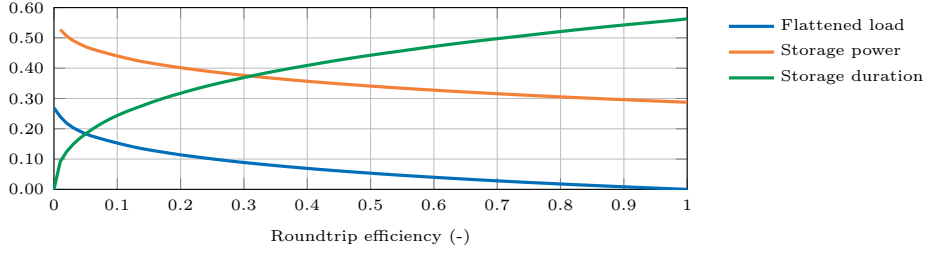
\begin{figure}[!t]
\centering
 \begin{tikzpicture}[font=\tiny]
    \begin{axis}[
        width=10cm, 
        height=4cm, 
        enlarge x limits=false,
        ymin = 0, ymax = 0.6,
        ytick distance = 0.1,
        yticklabel style={
        /pgf/number format/.cd,
        fixed, zerofill, precision=2
        },
        xticklabel style={
        /pgf/number format/1000 sep={}
    },
        scaled y ticks=false,
        xmin = 0, xmax = 1,
        xtick distance=0.1,
        xlabel={Roundtrip efficiency (-)},
        grid=both, 
        legend style={
        at={(1.05, 1)},
        anchor = north west,
        cells={anchor=west},
        draw=none,
        fill=none,
        font=\tiny
        }, 
    ]
    \addplot+[no marks, solid, opacity=1, very thick, NavyBlue] table [x=nsq, y=new_peak, col sep=comma] {data/potential.txt};
    \addplot+[no marks, solid, opacity=1, very thick, Orange] table [x=nsq, y=power, col sep=comma] {data/potential.txt};
    \addplot+[no marks, solid, opacity=1, very thick, ForestGreen] table [x=nsq, y=duration, col sep=comma] {data/potential.txt};
    \legend{
         Flattened load, Storage power, Storage duration 
    }
    \end{axis}
    \end{tikzpicture}
    \caption{Flattened load (normalized by the deviation from average load), storage power (normalized by average load), and storage duration (normalized by 12~hours).}
    \label{fig:potential}
\end{figure}

\subsection{Sources and assumptions}\label{sec:sources}
The data used in the case study is obtained from the following sources:
\begin{itemize}
    \item \textbf{Battery cost}: NREL's cost projections for utility-scale battery storage \citep{nrel2019update, nrel2020update, nrel2021update, nrel2025update, nrel2023update};
    \item \textbf{Inflation}: A consumer price index (\url{https://www.rateinflation.com/consumer-price-index/usa-historical-cpi/}) to adjust price data from different years;
    \item \textbf{Existing distribution-grid batteries in the Northeastern US}:
    \begin{itemize}
        \item Nantucket, MA, with an estimated investment cost of $\$33$~million for the 6MW/48MWh battery commissioned in 2019, $\$35.6$~million for the combustion turbine generator~\citep[p.~3.3]{balducci2019nantucket}, and total actual costs of $\$81$~million~\citep{gheorghiu2019nationalgrid};
        \item Provincetown, MA, with a reported investment cost of~\$54.8~million for a 25MVA/38MWh battery commissioned in 2022 \citep[Fig.~34]{eversource2024gridmodernization};
        \item Ponoma, NY, with a reported cost of $\$9.2$~million for a 3MW/12MWh battery commissioned in 2020~\citep[p.~3]{or2021cost}.
    \end{itemize}
    \item \textbf{Electricity price and demand for Nantucket}: The ISO-NE's webportal (\url{https://www.iso-ne.com/isoexpress/web/reports/load-and-demand/-/tree/nodal-load-weights}) from network node \texttt{LD.CANDLE 13.2} with ID \texttt{16255} in load zone \texttt{4006} for hourly load and price profiles for the year 2024. We assume that the price profile is identical in each planning period. In reality, it will be different and efforts are made to project future price patters, \eg, via NREL \href{https://www.nrel.gov/analysis/cambium.html}{Cambium}. We do not use these because they currently underestimate price variability~\citep[p.~5]{seel2021cambium}, an important determinant of arbitrage profitability. In addition, we assume that future load follows the same pattern as in~2024 and is linearly scaled by the increase in peak load. This is unrealistic because we expect shifts in electricity usage patterns. While distribution companies project these shifts, the underlying data was not included in National Grid's electric sector modernization plan;
    \item \textbf{Peak load projections through 2050}: National Grid's electric sector modernization plan~\citep[2023 to 2050 Electric Peak (MW) Forecast, p.~62]{nationalgrid2024ESMP}. We consider a single load scenario, multiple scenarios could be considered via stochastic programming~\citep{basciftci2024adaptive};
    \item \textbf{Backup generation efficiency}: Caterpillar (\url{https://s7d2.scene7.com/is/content/Caterpillar/CM20150703-52095-43744}) reports a heat rate of 10.4 MJ/kW-hr at temperatures of around 30$^\circ$C for their combustion turbine generators. This heat rate is accurate during hot summer afternoons, \ie, when backup generation is needed most; 
    \item \textbf{Diesel fuel costs for backup generation}: The Energy Information Administration reports \$4 per gallon, \ie, about \$1.057 per liter as of 9 September 2025 (\url{https://www.eia.gov/petroleum/gasdiesel/});
    \item \textbf{The lower heating value of diesel fuel}: An online engineering manual reports 36MJ/l (\url{https://www.engineeringtoolbox.com/fuels-higher-calorific-values-d_169.html});
    \item \textbf{Capacity prices and scarcity events}: The ISO-NE webportal for capacity prices (\url{https://www.iso-ne.com/about/key-stats/markets#fcaresults}) and scarcity events (\url{https://www.iso-ne.com/isoexpress/web/reports/auctions/-/tree/fcm-hist-csc});
    \item \textbf{Effective load carrying capacities}: ISO-NE's market monitor reports effective load carrying capacities, \ie, discount factors applied to resources participating in capacity markets, of~1 for storage with a duration of 2h or longer and for combustine turbine backup generation~\citep[p.~64]{potomac2022isonemonitor}, which means that these resources can be credited for their full nominal power supply capacity. 
\end{itemize}

Table~\ref{tab:cs_params} lists all case study parameters with references.

{%
\centering
\tiny
\begin{longtable}{llcrl}
    \caption{Case study parameters.} \\
    \label{tab:cs_params}
    Parameter & Resource & Symbol & Value & Reference/Note \\
    \midrule
    \endhead
    \bottomrule
    \endlastfoot
    Supply resources & & $\set{R}$ & $\{ \mathrm{b}, \mathrm{g}, \mathrm{s}\}$ & b: backup, g: grid, s: storage.\\
    Demand resources & & $\set{D}$ & $\{ \mathrm{g}, \ell, \mathrm{s} \}$ & g: grid, $\ell$: electric load, s: storage.\\
    Planning periods & & $\set{N}$ & $\{2025, \ldots, 2050\}$ & Yearly resolution. \\
    \# of planning periods & & $N$ & $\vert \set{N} \vert$ \\
    Contingency cases & & $\set{C}$ & $\{0,1\}$ & 0: no contingency, 1: largest installed cable fails. \\
    \# of contingencies & & $C$ & $\vert C \vert$ \\
    Operating periods & & $\set{J}$ & $\{1, \ldots, 365\}$ & Days per planning period. \\
    \# of operating periods & & J & $\vert \set{J} \vert$ \\
    Operating subperiods & & $\set{K}$ & $\{1,\ldots,24\}$ & Hours per operation period.\\
    \# of subperiods & & $K$ & $\vert \set{K} \vert $ \\
    Time discretization & & $\Delta t$ & 1h & Resolution of available ISO-NE load and price data.\\
    Time without contingency & & $\bm T_\mathrm{0}$ & 0.8h & Per subperiod. \\
    Time with contingeny & & $\bm T_\mathrm{1}$ & 0.2h & Per subperiod. \\  
    Discount rate & & & 0 & \\
    Maximum demand & Load & $\bar y_\ell$ & Fig.\ref{fig:Ntkt2024} & Hourly load (Fig.\ref{fig:Ntkt2024}) is scaled linearly. \\
    &&&& with yearly peak load evolution (Fig.\ref{fig:nantucket}).\\
    Charging efficiency & Storage &  $\eta^\mathrm{c}$ & 0.913 & \citep[p.3.3]{balducci2019nantucket} \\
    Discharging efficiency & Storage & $\eta^\mathrm{d}$ & 0.913 & \citep[p.3.3]{balducci2019nantucket} \\
    Maximum discharge cycles & Storage & $C^\mathrm{s}$ & 150 cycles/yr & \citep[p.3.2]{balducci2019nantucket}\\  
    Duration & Storage & $T^\mathrm{s}$ & 8h & Same as existing battery~\citep[p.3.2]{balducci2019nantucket}.\\
    Minimum investment & Backup & $\ubar x_\mathrm{b}$ &  2MW & Assumed similar threshold to storage.\\
    & Grid & $\ubar x_\mathrm{g}$ & 40MW & Existing cables are 36MW and 38MW installed in 1996 and 2006,  \\
    &&&& \citep{bluedotliving2024nantucket}. \\
    & Storage & $\ubar x_\mathrm{s}$ & 2MW & \cite{or2021cost} installed a 3MW battery.\\
    Maximum investment & Backup & $\bar x_\mathrm{b}$ & 30MW & Assumed to be twice the currently installed backup capacity.\\
    & Grid & $\bar x_\mathrm{g}$ & 40MW & Assumed to be the same as the minimum investment. \\
    & Storage & $\bar x_\mathrm{s}$ & 24MW & Maximum peak shaving potential (Sec.\ref{sec:max_peak_shaving}). \\
    Existing units & Backup & $\set{I}_\mathrm{b}$ & $\{1\}$ & \citep[p.3.4]{balducci2019nantucket} \\  
    & Grid & $\set{I}_\mathrm{g}$ & $\{1,2\}$ & \citep{bluedotliving2024nantucket} \\
    & Storage & $\set{I}_\mathrm{s}$ & $\{1\}$ & \citep[p.3.2]{balducci2019nantucket} \\
    Lifetime & Backup & $N_\mathrm{b}$ & 20yr & \citep[p.3.5]{balducci2019nantucket} \\
    & Grid & $N_\mathrm{g}$ & 40yr & \citep{marthasvineyard2023faulty} \\
    & Storage & $N_\mathrm{s}$ & 20yr & \citep[p.3.2]{balducci2019nantucket} \\
    Existing capacity & Backup & $\bm x^0_\mathrm{b}$ & Fig.\ref{fig:cs_params} & \citep[p.3.4]{balducci2019nantucket} + lifetime assumption.\\
    & Grid & $\bm x^0_\mathrm{g}$ & Fig.\ref{fig:cs_params} & \citep{bluedotliving2024nantucket} + lifetime assumption.\\
    & Storage & $\bm x^0_\mathrm{s}$ & Fig.\ref{fig:cs_params} & \citep[p.3.2]{balducci2019nantucket} + lifetime assumption.\\
    Investment costs & Backup & $\bm p_\mathrm{b}$ & Fig.\ref{fig:cs_params} &  PNNL and actual project costs (Sec.\ref{sec:sources}).\\ 
    & Grid & $\bm p_\mathrm{g}$ & Fig.\ref{fig:cs_params} & 
    40MW cable would have cost $\$200$ million in 2019~\citep{gheorghiu2019nationalgrid}. \\
    & Storage & $\bm p_\mathrm{s}$ & Fig.\ref{fig:cs_params} & PNNL and actual project costs (Sec.\ref{sec:sources}).\\
    Investment fix costs & Backup & $\bm p^0_\mathrm{b}$ & 0 & Considered through minimum investment threshold.\\ 
    & Grid & $\bm p^0_\mathrm{g}$ & 0 & Considered through minimum investment threshold.\\
    & Storage & $\bm p^0_\mathrm{s}$ & 0 & Considered through minimum investment threshold. \\
    Capacity prices & Backup & $\bar{\bm p}_\mathrm{b}$ & Fig.\ref{fig:fcm_price} & \url{https://www.iso-ne.com/about/key-stats/markets#fcaresults}. \\
    & Grid & $\bar{\bm p}_\mathrm{g}$ & 0 & Capacity payment applies only to generation. \\
    & Storage & $\bar{\bm p}_\mathrm{s}$ & Fig.\ref{fig:fcm_price} & \url{https://www.iso-ne.com/about/key-stats/markets#fcaresults}. \\
    Supply cost & Backup & $\bm p^\mathrm{s}_\mathrm{b}$ & \$305/MWh & See Sec.\ref{sec:sources} and Fig.\ref{fig:cs_params} \\
    & Grid & $\bm p^\mathrm{s}_\mathrm{g}$ & Figs.\ref{fig:Ntkt2024}\&\ref{fig:cs_params} & Same as the 2024 cost, available from ISO-NE (Sec.\ref{sec:sources}).\\
    & Storage & $\bm p^\mathrm{s}_\mathrm{s}$ & 0 & Could be set to nonzero to account for degradation.\\
    Demand revenue & Grid & $\bm p^\mathrm{d}_\mathrm{g}$ & Fig.\ref{fig:Ntkt2024} & Same as supply cost.\\
    & Load & $\bm p^\mathrm{d}_\ell$ & 0 & if load shedding is not allowed,\\
    &&& \$$9,337$/MWh& if load shedding is allowed \citep[p.70]{potomac2025isonemonitor}.\\
    & Storage & $\bm p^\mathrm{d}_\mathrm{s}$ & 0 & Could be set to nonzero to account for degradation.\\
    Maximum solution time & & & $14,400$s & \url{docs.gurobi.com/projects/}, \\
    &&&& \url{optimizer/en/current/reference/parameters.html\#timelimit}. \\
    Maximum MIPGap & & & $10^{-5}$ & Relative mixed-integer optimality gap: \url{docs.gurobi.com/projects/}, \\
    &&&& \url{optimizer/en/current/reference/parameters.html\#mipgap}. \\
\end{longtable}
}%

\setcounter{equation}{0}
\setcounter{Prop}{0}
\setcounter{table}{0}
\setcounter{figure}{0}
\section{Optimization problem}\label{sec:opt_formulation}
\subsection{Decision variables}
{%
\centering
\tiny
\begin{longtable}{lllll}
    \caption{Decision variables.} \\
    \label{tab:dec_vars}
    Variable & Symbol & Space & Dimension & Note \\
    \midrule
    \endhead
    Supply investment & $\bm x$ & $\mathbb{R}^{ \vert \set{R} \vert N}_+$ & Power & Capacity of type $r \in \set{R}$ becoming available at the start of \\
    &&&&planning period $n \in \set{N}$. \\
    Largest investment & $\bm x^\mathrm{max}$ & $\mathbb{R}^{N}_+$ & Power & Largest grid investment that is still live during planning period $n \in \set{N}$. \\
    Investment indicator & $\bm z$ & $\{0,1\}^{\vert \set{R} \vert N}$ &  None & $=1$ if $x_{rn} > 0$ for $r \in \set{R}$ and $n \in \set{N}$, $=0$ otherwise.\\
    Supply capacity & $\bm x^\mathrm{tot}$ & $\mathbb{R}^{\vert \set{R} \vert N}_+$ & Power & Installed capacity of type $r \in \set{R}$ available during planning period $n \in \set{N}$.\\
    Supply operation & $\bm y^\mathrm{s}$ & $\mathbb{R}^{\vert \set{R} \vert NJKC}_+$ & Power & Supply of type $r \in \set{R}$ during subperiod $k \in \set{K}$, in operating\\
    &&&& period $j \in \set{J}$, in planning period $n \in \set{N}$, in contingency case $c \in \set{C}$.\\
    Demand operation & $\bm y^\mathrm{d}$ & $\mathbb{R}^{\vert \set{D} \vert NJKC}_+$ & Power & Same as above for demand.\\
    State-of-charge & $\bm y$ & $\mathbb{R}^{NJKC}_+$ & Energy & State-of-charge at the beginning of subperiod $k \in \set{K}$, in operating\\
    &&&& period $j \in \set{J}$, in planning period $n \in \set{N}$, in contingency case $c \in \set{C}$.\\
    State-of-charge target & $\bm y^0$ & $\mathbb{R}^N_+$ & Energy & State-of-charge target in planning period~$n \in \set{N}$. \\
    Operating indicator & $\bm z^\mathrm{M}$ & $\mathbb{R}^{NJKC}$ & None & $=1$ if $y^\mathrm{d}_{\ell njkc} > x^\mathrm{tot}_{\mathrm{g}nc}$ in subperiod $k \in \set{K}$, in operating period $j \in \set{J}$, \\
    &&&&  in planning period $n \in \set{N}$, in contingency case $c \in \set{C}$; $= 0$ otherwise.\\
    \bottomrule
    \endlastfoot
\end{longtable}
}%
\subsection{Auxiliary functions}
\begin{equation*}
    \ubar n(n, N) = \max\{1, n - N + 1\}
\end{equation*}
\subsection{Full formulation}
We introduce epigraphical variables for the capacity functions~$\bar x(\cdot)$ and state-of-charge variables to increase sparsity at the expense of a greater number of decision variables and constraints.
\subsubsection{With full market participation}
\begin{mini!}|s|[2]<b>
    {}
    {\sum_{n \in \set{N}} \left( \sum_{r \in \set{R}} p_{rn} x_{rn} + p^0_{rn} z_{rn} - \bar p_{rn} x^\mathrm{tot}_{rn} \right) + \sum_{c \in \set{C}} T_{c} \sum_{j \in \set{J}} \sum_{k \in \set{K}} \left( \sum_{r \in \set{R}} p^\mathrm{s}_{rnjk} y^\mathrm{s}_{rnjkc} - \sum_{r \in \set{D}} p^\mathrm{d}_{rnjk} y^\mathrm{d}_{rnjkc} \right)\label{P:obj}}
    {\label{pb:P}}
    {}
    \addConstraint{
        \bm x_{r}, \bm x^\mathrm{tot}_r \in \mathbb{R}^N_+, \bm y^\mathrm{s}_r \in \mathbb{R}^{N J K C}_+, \bm z_r \in \{0,1\}^N,
    }{}{
    ~ \forall r \in \set{R},
    \label{P:varsR}
    }
    \addConstraint{
        \bm y^\mathrm{d}_r \in \mathbb{R}^{N J K C}_+,}{}{
        ~ \forall r \in \set{D},
        \label{P:varsD}
        } 
    \addConstraint{
        \bm x^\mathrm{max} \in \mathbb{R}^{N}_+, \, \bm y \in \mathbb{R}^{N J K C}_+, \bm y^0 \in \mathbb{R}^{N}_+,
    }{}{
    \label{P:varsC}
    }
    \addConstraint{
        \ubar x_r z_{rn} \leq x_{rn} \leq \bar x_r z_{rn},
    }{}{
    ~ \forall (r,n) \in \set{R} \times \set{N},
    \label{P:coninv}
    }
    \addConstraint{
        x^\mathrm{tot}_{rn} = \sum_{i \in \set{I}_r} x^0_{rni} + \sum_{i = \ubar n(n, N_r)}^n x_{ri},
    }{}{
    ~ \forall (r,n) \in \set{R} \setminus \{\mathrm{g}\} \times \set{N},
    \label{P:concap1}
    }
    \addConstraint{
        x^\mathrm{tot}_{\mathrm{g}nc} = \sum_{i \in \set{I}_\mathrm{g}} x^0_{\mathrm{g}ni} + \sum_{i = \ubar n(n, N_\mathrm{g})}^n x_{\mathrm{g}i} - c x^\mathrm{max}_{n},
    }{}{
    ~ \forall \{n,c\} \in \set{N} \times \set{C},
    \label{P:concap2}
    }
    \addConstraint{
        x^\mathrm{max}_{n} \geq x_{\mathrm{g}i},
    }{}{
    ~ \forall n \in \set{N}, \forall i \in \{\ubar n(n, N_\mathrm{g}), \ldots, n\},
    \label{P:conepi1}
    }
    \addConstraint{
    x^\mathrm{max}_{n} \geq  x^0_{\mathrm{g}ni},
    }{}{
    ~ \forall (n,i) \in \set{N} \times \set{I}_g,
    \label{P:conepi2}
    }
    \addConstraint{
    \sum_{r \in \set{R}} y^\mathrm{s}_{rnjkc} = \sum_{r \in \set{D}} y^\mathrm{d}_{rnjkc},
    }{}{
    ~ \forall (n,j,k,c) \in \set{N} \times \set{J} \times \set{K} \times \set{C},
    \label{P:conbalance}
    }
    \addConstraint{
    y^\mathrm{s}_{rnjkc} \leq x^\mathrm{tot}_{rn},
    }{}{
    ~\forall (r,n,j,k,c) \in \set{R} \setminus \{\mathrm{g}\} \times \set{N} \times \set{J} \times \set{K} \times \set{C},
    \label{P:consup}
    }
    \addConstraint{
    y^\mathrm{d}_{\mathrm{g}njkc} \leq x^\mathrm{tot}_{\mathrm{g}nc},~
    y^\mathrm{s}_{\mathrm{g}njkc} \leq x^\mathrm{tot}_{\mathrm{g}nc},
    }{}{
    ~ \forall (n,j,k,c) \in \set{N} \times \set{J} \times \set{K} \times \set{C},
    \label{P:congrid}
    }
    \addConstraint{
    y^\mathrm{d}_{\ell njkc} \leq \bar y_{\ell njk},~
    y^\mathrm{d}_{\mathrm{s} njkc} \leq x^\mathrm{tot}_{\mathrm{s}n},
    }{}{
    ~ \forall (n,j,k,c) \in \set{N} \times \set{J} \times \set{K} \times \set{C},
    \label{P:condem}
    }
    \addConstraint{
    0 \leq y^0_n \leq T^\mathrm{s} x^\mathrm{tot}_{\mathrm{s}n},~
    0 \leq y_{njkc} \leq T^\mathrm{s} x^\mathrm{tot}_{\mathrm{s}n},
    }{}{
    ~ \forall (n, j, k, c) \in \set{N} \times \set{J} \times \set{K} \times \set{C},
    \label{P:cony0}
    }
    \addConstraint{
    y_{nj1c} = y^0_{n} 
    + \Delta t \big( \eta^c y^\mathrm{d}_{\mathrm{s}nj1c} - \frac{y^\mathrm{s}_{\mathrm{s}nj1c}}{\eta^\mathrm{d}}\big) ,
    }{}{
    ~ \forall (n,j,c) \in \set{N} \times \set{J} \times \set{C},
    \label{P:cony1}
    }
    \addConstraint{
    y_{njkc} = y_{nj(k-1)c} 
    + \Delta t \big( \eta^c y^\mathrm{d}_{\mathrm{s}njkc} - \frac{y^\mathrm{s}_{\mathrm{s}njkc}}{\eta^\mathrm{d}}\big) ,
    }{}{
    ~ \forall (n,j,k,c) \in \set{N} \times \set{J} \times \set{K} \setminus\{1\} \times \set{C},
    \label{P:cony+}
    }
    \addConstraint{
    y_{njKc} = y^0_n,
    }{}{
    ~ \forall (n,j,c) \in \set{N} \times \set{J} \times \set{C},
    \label{P:conter}
    }
    \addConstraint{
    \frac{\Delta t}{\eta^\mathrm{d}} \sum_{j \in \set{J}} \sum_{k \in \set{K}} y^\mathrm{s}_{\mathrm{s}njkc} \leq C^\mathrm{s} T^\mathrm{s} x^\mathrm{tot}_{\mathrm{s}n},
    }{}{
    ~\forall (n,c) \in \set{N} \times \set{C}.
    }
\end{mini!}

\subsubsection{Without load shedding}
Replace constraint~\eqref{P:condem} by
\begin{equation*}
    y^\mathrm{d}_{\ell njkc} = \bar y^\mathrm{d}_{\ell njk},
    ~~
    \forall (n,j,k,c) \in \set{N} \times \set{J} \times \set{K} \times \set{C}.
\end{equation*}

\subsubsection{With market participation constraints}
For any $(n,j,k,c) \in \set{N} \times \set{J} \times \set{K} \times \set{C}$, we limit the supply from non-grid resources to the shortfall of grid capacity from load, \ie,
\begin{equation}\label{eq:m_reformulated}
    \sum_{r \in \set{R} \setminus \{\mathrm{g}\}} y^\mathrm{s}_{rnjkc} \leq [y^\mathrm{d}_{\ell njkc} - x^\mathrm{tot}_{\mathrm{g}nc}]^+.
\end{equation}
The difference to the original market participation constraint~\eqref{eq:m} is that we have replaced $\bar x_{\mathrm{g}nc}(\bm x_\mathrm{g})$ by the auxiliary variable $x^\mathrm{tot}_{\mathrm{g}nc}$. The case distinction in the $[\cdot]^+$ term can be handled with disjunctive constraints. We will determine the big-M parameters for these constraints under the following assumption.

\begin{Ass}\label{ass:buildout}
    The build-out of any supply resource is limited by local electricity demand.
\end{Ass}

Assumption~\ref{ass:buildout} is in line with the spirit of market participation constraints, separating electricity generation from distribution. To apply Assumption~\ref{ass:buildout}, we introduce the following lemma.

\begin{lem}\label{lem:bounds} For any optimization problem of the form
\begin{equation*}
    \min f_0(\bm x)~~\text{s.t.}~~\sum_{j \in \set{J}} x_j \geq x_0,~ \bm x \in \{\bm 0\} \cup [\ubar{ \bm x}, \bar{ \bm x}]^J,
\end{equation*}
    where $f_0$ is an increasing function, $x_0 \in \mathbb{R}$, $\ubar{\bm x} \in \mathbb{R}^J$, and $\bar{\bm x} \in \mathbb{R}^J$ are constants, and $\set{J}$ is a subset of~$\{1,\ldots,J\}$, all optimal solutions satisfy 
    \begin{equation*}
        \sum_{j \in \set{J}} x^\star_j < x_0 + \bar x.
    \end{equation*}
\end{lem}

\begin{proof}
    We prove the claim by contradiction. Assume that there was an optimal solution $\bm x'$ such that $\sum_{j \in \set{J}} x'_j \geq x_0 + \bar x$. Set any nonzero component of $x'$ to zero. The modified solution is feasible because $\bm x' \leq \bar{\bm x}$, and has a better objective value because $f_0$ is increasing. Thus, the original solution cannot be optimal.
\end{proof}

We now state a mixed-integer linear reformulation of the market participation constraints.

\begin{Prop}\label{prop:supply_limit}
The supply limit can be modeled with auxiliary binary variables, \ie,
    \begin{align*}
    & \sum_{r \in \set{R} \setminus \{\mathrm{g}\}} y^\mathrm{s}_{rnjkc} \leq [y^\mathrm{d}_{\ell njkc} - x^\mathrm{tot}_{\mathrm{g}nc}]^+\\
    \iff & 
    \exists \, z^\mathrm{M}_{njkc} \in \{0,1\}: 
    \begin{cases}
        (1 - z^\mathrm{M}_{njkc}) \underline{M}_1 \leq  y^\mathrm{d}_{\ell njkc} - x^\mathrm{tot}_{\mathrm{g}nc} \leq z^\mathrm{M}_{njkc} \overline{M}_{1njkc} \\
        \sum_{r \in \set{R} \setminus \{\mathrm{g}\}} y^\mathrm{s}_{rnjkc} \leq y^\mathrm{d}_{\ell njkc} - x^\mathrm{tot}_{\mathrm{g}nc} - (1 - z^\mathrm{M}_{njkc}) \underline{M}_2 \\
        \sum_{r \in \set{R} \setminus \{\mathrm{g}\}} y^\mathrm{s}_{rnjkc} \leq z^\mathrm{M}_{njkc} \overline{M}_{2njk},
    \end{cases}
\end{align*}
where 
$-\underline{M}_1 = -\underline{M}_2 = \max\{ \bar y_\ell \} + 2 \bar x_\mathrm{g}$, 
$\overline{M}_{1njkc} = \bar y_{\ell njk} - \bar x_{\mathrm{g}nc}(\bm 0)$,
$\overline{M}_{2njk} = \bar y_{\ell njk}$, and~$\max\{\bar y_\ell\}$ returns the maximum electricity demand over all planning and operating periods. 
\end{Prop}

\begin{proof}[Proof of Proposition~\ref{prop:supply_limit}]
We first show the $(\implies)$ and then the $(\impliedby)$ direction.

Let $y^\mathrm{d}_{\ell njkc} - x^\mathrm{tot}_{\mathrm{g}nc}\geq0$, then the inequalities
$(1 - z^\mathrm{M}_{njkc}) \underline{M}_1 \leq  y^\mathrm{d}_{\ell njkc} - x^\mathrm{tot}_{\mathrm{g}nc} \leq \overline{M}_{1njkc} z^\mathrm{M}_{njkc}$ are valid if $z^\mathrm{M}_{njkc} = 1$ and
$\overline{M}_{1njkc} \geq y^\mathrm{d}_{\ell njkc} - x^\mathrm{tot}_{\mathrm{g}nc}$ for all feasible $y^\mathrm{d}_{\ell njkc}$ and $x^\mathrm{tot}_{\mathrm{g}nc}$. Thus,
\begin{align*}
    & \max \, y^\mathrm{d}_{\ell njkc} - x^\mathrm{tot}_{\mathrm{g}nc}~~\text{s.t.}~~(y^\mathrm{d}_{\ell njkc}, x^\mathrm{tot}_{\mathrm{g}nc})~~\text{feasible in \eqref{pb:P}} \\
    \leq~&\max \big\{ y^\mathrm{d}_{\ell njkc}~~\text{s.t.}~~
    y^\mathrm{d}_{\ell njkc}~~\text{feasible in \eqref{P:varsC}, \eqref{P:condem}}
    \big\} \\
    & - \min \big\{ x^\mathrm{tot}_{\mathrm{g}nc}~~\text{s.t.}~~x^\mathrm{tot}_{\mathrm{g}nc}~~\text{feasible in \eqref{P:varsR}, \eqref{P:coninv}, \eqref{P:concap2}}\big\} \\
    \leq~&\bar y_{\ell njk} - \bar x_{\mathrm{g}nc}(\bm 0) = \overline{M}_{1njkc},
\end{align*}
where the first inequality holds because we split the initial problem into two relaxed subproblems and the second inequality holds because $\bar x_{\mathrm{g}nc}$ is nondecreasing in~$\bm x_\mathrm{g}$ and $\bm x_\mathrm{g} \geq 0$. For the remaining inequalities, 
\begin{equation*}
    \sum_{r \in \set{R} \setminus \{\mathrm{g}\}} y^\mathrm{s}_{rnjkc} \leq y^\mathrm{d}_{\ell njkc} - x^\mathrm{tot}_{\mathrm{g}nc} - (1 - z^\mathrm{M}_{njkc}) \underline{M}_2
\end{equation*}
is implied by $\sum_{r \in \set{R} \setminus \{\mathrm{g}\}} y^\mathrm{s}_{rnjkc} \leq [y^\mathrm{d}_{\ell njkc} - x^\mathrm{tot}_{\mathrm{g}nc}]^+$ for $z^\mathrm{M}_{njkc} = 1$, and
\begin{equation*}
    \sum_{r \in \set{R} \setminus \{\mathrm{g}\}} y^\mathrm{s}_{rnjkc} \leq z^\mathrm{M}_{njkc} \overline{M}_{2njk},
\end{equation*}
holds for $z^\mathrm{M}_{njkc} = 1$ because
\begin{align*}
    \sum_{r \in \set{R}\setminus\{ \mathrm{g} \}} y^\mathrm{s}_{rnjkc} 
    \leq
    \big[ y^\mathrm{d}_{\ell njkc} - x^\mathrm{tot}_{\mathrm{g}nc} \big]^+ \leq \bar y_{\ell njk} = \overline{M}_{2njk}.
\end{align*}
The second inequality holds thanks to the constraints~\eqref{P:varsC} and~\eqref{P:condem}.

If $y^\mathrm{d}_{\ell njkc} - x^\mathrm{tot}_{\mathrm{g}nc}<0$, the inequalities $(1 - z^\mathrm{M}_{njkc}) \underline{M}_1 \leq  y^\mathrm{d}_{\ell njkc} - x^\mathrm{tot}_{\mathrm{g}nc} \leq \overline{M}_{1njkc} z^\mathrm{M}_{njkc}$ are valid for $z^\mathrm{M}_{njkc} = 0$ and
\begin{align*}
    & \min \, y^\mathrm{d}_{\ell njkc} - x^\mathrm{tot}_{\mathrm{g}nc}~~\text{s.t.}~~(y^\mathrm{d}_{\ell njkc}, x^\mathrm{tot}_{\mathrm{g}nc})~~\text{feasible in \eqref{pb:P}} \\
    \geq~& \min \big\{ y^\mathrm{d}_{\ell njkc}
    ~~\text{s.t.}~~ 
    y^\mathrm{d}_{\ell njkc}~~\text{feasible in \eqref{P:varsC}, \eqref{P:condem}}
    \big\} \\
    & - \max \big\{ x^\mathrm{tot}_{\mathrm{g}nc}~~\text{s.t.}~~x^\mathrm{tot}_{\mathrm{g}nc}~~\text{feasible in \eqref{P:varsR}, \eqref{P:coninv}, \eqref{P:concap2}}\big\} \\
    =~& 0 -\max \big\{ x^\mathrm{tot}_{\mathrm{g}nc}~~\text{s.t.}~~x^\mathrm{tot}_{\mathrm{g}nc}~~\text{feasible in \eqref{P:varsR}, \eqref{P:coninv}, \eqref{P:concap2}}\big\} \\
    \geq~& - \max\{\bar y_\ell\} - 2 \bar x_{\mathrm{g}} 
    =
    \underline{M}_1,
\end{align*}
where the inequality holds again because we split the initial problem into two relaxed subproblems and the second inequality follows from Lemma~\ref{lem:bounds}, which applies with constant $x_0 = \max\{ \bar y_\ell \} + \bar x_\mathrm{g}$ thanks to Assumption~\ref{ass:buildout}. For the remaining inequalities, 
\begin{equation*}
    \sum_{r \in \set{R} \setminus \{\mathrm{g}\}} y^\mathrm{s}_{rnjkc} \leq z^\mathrm{M}_{njkc} \overline{M}_{2njk},
\end{equation*}
is implied by $\sum_{r \in \set{R} \setminus \{\mathrm{g}\}} y^\mathrm{s}_{rnjkc} \leq [y^\mathrm{d}_{\ell njkc} - x^\mathrm{tot}_{\mathrm{g}nc}]^+$ for $z^\mathrm{M}_{njkc} = 0$, and
\begin{equation*}
    \sum_{r \in \set{R} \setminus \{\mathrm{g}\}} y^\mathrm{s}_{rnjkc} \leq y^\mathrm{d}_{\ell njkc} - x^\mathrm{tot}_{\mathrm{g}nc} - (1 - z^\mathrm{M}_{njkc}) \underline{M}_2
\end{equation*}
holds for $z^\mathrm{M}_{njkc} = 0$ because
\begin{align*}
    & \min \, y^\mathrm{d}_{\ell njkc} - x^\mathrm{tot}_{\mathrm{g}nc} - \sum_{r \in \set{R} \setminus \{\mathrm{g}\}} y^\mathrm{s}_{rnjkc}
    ~~\text{s.t.} ~~ (x^\mathrm{tot}_{\mathrm{g}nc}, y^\mathrm{d}_{\ell njkc}, y^\mathrm{s}_{\mathrm{b}njkc}, y^\mathrm{s}_{\mathrm{s}njkc})~~\text{feasible in \eqref{pb:P}} \\
    =~& \min \, y^\mathrm{d}_{\ell njkc} - x^\mathrm{tot}_{\mathrm{g}nc}~~\text{s.t.}~~(y^\mathrm{d}_{\ell njkc}, x^\mathrm{tot}_{\mathrm{g}nc})~~\text{feasible in \eqref{pb:P}} \\
    \geq~& - \max\{\bar y_\ell\} - 2 \bar x_{\mathrm{g}} 
    =
    \underline{M}_2,
\end{align*}
where the first equality holds because $0 \leq \sum_{r \in \set{R} \setminus \{\mathrm{g}\}} y^\mathrm{s}_{rnjkc} \leq [y^\mathrm{d}_{\ell njkc} - x^\mathrm{tot}_{\mathrm{g}nc}]^+ = 0$ and the inequality follows from the same reasoning as in the derivation for~$\underline{M}_2$.

We now prove the reverse implication. Let $z^\mathrm{M}_{njkc} = 1$, then the following constraints apply
\begin{equation*}
    0 \leq y^\mathrm{d}_{\ell njkc} - x^\mathrm{tot}_{\mathrm{g}nc} \leq \overline{M}_{1njkc},~
    \sum_{r \in \set{R} \setminus \{\mathrm{g}\}} y^\mathrm{s}_{rnjkc} \leq \min\left\{ y^\mathrm{d}_{\ell njkc} - x^\mathrm{tot}_{\mathrm{g}nc}, \, \overline{M}_{2njk} \right\}.
\end{equation*}
Following the same steps as in the first part of the proof, we see that the constraints involving $\overline{M}_{1njkc}$ and $\overline{M}_{2njk}$ are redundant. Thus,
\begin{equation*}
    z^\mathrm{M}_{njkc} = 1 \implies
    y^\mathrm{d}_{\ell njkc} - x^\mathrm{tot}_{\mathrm{g}nc} \geq 0, \,
    \sum_{r \in \set{R} \setminus \{\mathrm{g}\}} y^\mathrm{s}_{rnjkc} \leq y^\mathrm{d}_{\ell njkc} - x^\mathrm{tot}_{\mathrm{g}nc}.
\end{equation*}

Similarly, one can show that 
\begin{equation*}
    z^\mathrm{M}_{njkc} = 0 \implies y^\mathrm{d}_{\ell njkc} - x^\mathrm{tot}_{\mathrm{g}nc} \leq 0, \, \sum_{r \in \set{R} \setminus \{\mathrm{g}\}} y^\mathrm{s}_{rnjkc} \leq 0. \qedhere
\end{equation*}
\end{proof}
\setcounter{equation}{0}
\setcounter{Prop}{0}
\setcounter{table}{0}
\setcounter{figure}{0}
\section{Results}\label{apx:results}
{%
\centering
\tiny
\begin{longtable}{l|rrrrrrrrr}
    \caption{Case study results.} \label{tab:cs_results}\\
    \toprule
    \multicolumn{10}{c}{Experiments} \\
    \midrule
    Number & 1 & 2 & 3 & 4 & 5 & 6 & 7 & 8 & 9 \\
    \midrule
    \endhead
    \bottomrule
    \endlastfoot
    \multicolumn{7}{c}{Parameters} \\
    \midrule
    Market participation & Peak & $\cdots$ &  Full & $\cdots$ & $\cdots$ & $\cdots$ & $\cdots$ & Peak & Full \\
    Available investments & g & g+s & $\cdots$ & $\cdots$ & $\cdots$ & b+g+s & $\cdots$ & g+s & $\cdots$ \\
    Storage cost (\$/kWh) & na & 604 & $\cdots$ & $\cdots$ & $\cdots$ & $\cdots$ & $\cdots$ & 1 & $\cdots$ \\  
    Cycle limit & yearly & $\cdots$ & $\cdots$ &  $\cdots$ & daily & yearly & $\cdots$ & $\cdots$ & $\cdots$ \\
    Cap. price (\$/kW-month) & na & $\cdots$ & 0.000 & 3.064 & 0.000 & $\cdots$ & 3.064 & na & 0.000 \\
    \midrule
    \multicolumn{10}{c}{Solution quality} \\
    \midrule
    Total cost (M\$) & \textbf{678.874} & \textbf{647.026} & \textbf{633.441} & \textbf{616.688} & \textbf{659.310} & \textbf{629.823} & \textbf{599.411} & \textbf{609.227} & \textbf{432.174} \\
    Solve time (s) & 52.848 & 4,211.116 & 4,083.086 & 2,374.322 & \textbf{886.687} & 17,631.037 & 20,510.171 & 14,402.060 & 651.564 \\
    Maximum MIP gap (\%) & 0.001 & $\cdots$ & 0.000 & $\cdots$ & $\cdots$ & $\cdots$ & $\cdots$ & 0.004 & 0.000 \\
    \midrule
    \multicolumn{10}{c}{Costs (M\$)} \\
    \midrule
    Total operating & 331.384 & \textbf{331.245} & \textbf{317.661} & $\cdots$ & 313.509 & 331.632 & 331.633 & 331.065 & 151.047 \\
    - base case & 331.465 & $\cdots$ & 317.635 & $\cdots$ & 313.446 & 319.105 & $\cdots$ & 331.465 & 137.308  \\
    - contingeny & 331.064 & 330.367 & 317.763 & $\cdots$ & \textbf{313.762} & \textbf{381.742} & $\cdots$ & 329.466 & 205.999 \\
    Total capital & \textbf{347.490} & \textbf{315.781} & $\cdots$ & $\cdots$ & \textbf{345.801} & \textbf{298.191} & $\cdots$ & 278.162 & 281.127 \\
    - backup & na & $\cdots$ & $\cdots$ & $\cdots$ & $\cdots$ & \textbf{89.451} & $\cdots$ & na & $\cdots$  \\
    - grid & 347.490 & 277.992 & $\cdots$ & $\cdots$ & $\cdots$ & \textbf{183.222} & $\cdots$ & 277.992 & $\cdots$ \\
    - storage & na & 37.789 & $\cdots$ & $\cdots$ & \textbf{67.809} & 25.518 & $\cdots$ & 0.170 & 3.135 \\
    Total capacity payment & na & $\cdots$ & 0.000 & \textbf{-16.753} & 0.000 & $\cdots$ & \textbf{-30.413} & na & 0.000 \\
    - backup & na & $\cdots$ & 0.000 & -7.126 & 0.000 & $\cdots$ & \textbf{-23.138} & na & 0.000 \\
    - grid & na & $\cdots$ & $\cdots$ & $\cdots$ & $\cdots$ & $\cdots$ & $\cdots$ & $\cdots$ & $\cdots$ \\
    - storage & na & $\cdots$ & 0.000 & -9.627 & 0.000 & $\cdots$ & -7.275 & na & 0.000 \\
    \midrule
    \multicolumn{10}{c}{Investment decisions (MW)} \\
    \midrule
    Terminal capacity & 160.000 & 139.318 & $\cdots$ & $\cdots$ & 157.900 & 138.084 & $\cdots$ & 139.289 & 367.921 \\
    - backup & 0.000 & $\cdots$ & $\cdots$ & $\cdots$ & $\cdots$ & 6.513 & $\cdots$ & 0.000 & $\cdots$ \\
    - grid & 160.000 & 120.000 & $\cdots$ & $\cdots$ & $\cdots$ & $\cdots$ & $\cdots$ & $\cdots$ & $\cdots$ \\
    - storage & 0.000 & 19.318 & $\cdots$ & $\cdots$ & 37.900 & 11.571 & $\cdots$ & 19.289 & 247.921 \\
    Total investment & 160.000 & 139.318 & $\cdots$ & $\cdots$ & 157.900 & 153.832 & $\cdots$ & 141.290 & 511.921 \\
    - backup & na & $\cdots$ & $\cdots$ & $\cdots$ & $\cdots$ & 22.261 & $\cdots$ & na & $\cdots$ \\
    - grid & 160.000 & 120.000 & $\cdots$ & $\cdots$ & $\cdots$ & $\cdots$ & $\cdots$ & $\cdots$ & $\cdots$ \\
    - storage & na & \textbf{19.318} & $\cdots$ & $\cdots$ & \textbf{37.900} & \textbf{11.571} & $\cdots$ & \textbf{21.290} & \textbf{391.921} \\
    \midrule
    \multicolumn{10}{c}{Operating decisions: Base case (GWh/yr)} \\
    \midrule
    Demand (w/o storage) & 295.000 & $\cdots$ & $\cdots$ & $\cdots$ & 295.249 & 295.012 & $\cdots$ & 295.000 & 483.825 \\
    - grid & 0.000 & $\cdots$ & $\cdots$ & $\cdots$ & \textbf{0.248} & 0.012 & $\cdots$ & 0.000 & 188.825 \\
    - load & 295.000 & $\cdots$ & $\cdots$ & $\cdots$ & $\cdots$ & $\cdots$ & $\cdots$ & $\cdots$ & $\cdots$ \\
    - storage & 0.000 & $\cdots$ & 13.237 & $\cdots$ & 20.358 & 10.003 & $\cdots$ & 0.000 & 307.816 \\
    Supply (w/o storage) & 295.000 & $\cdots$ & 297.203 & $\cdots$ & 298.637 & 296.677 & $\cdots$ & 295.000 & 535.026 \\
    - backup & 0.000 & $\cdots$ & 0.075 & $\cdots$ & $\cdots$ & 0.242 & $\cdots$ & 0.000 & 0.029 \\
    - grid & 295.000 & $\cdots$ & 297.129 & $\cdots$ & 298.562 & 296.435 & $\cdots$ & 295.000 & 535.026 \\
    - storage & 0.000 & $\cdots$ & \textbf{11.034} & $\cdots$ & 16.970 & 8.338 & $\cdots$ & 0.000 & 256.586\\
    \midrule
    \multicolumn{10}{c}{Operating decisions: Contingency (GWh/yr)} \\
    \midrule
    Demand (w/o storage) & 295.000 & $\cdots$ & $\cdots$ & $\cdots$ & 295.249 & 295.012 & $\cdots$ & 295.000 & 405.713\\
    - grid & 0.000 & $\cdots$ & $\cdots$ & $\cdots$ & 0.248 & 0.012 & $\cdots$ & 0.000 & 110.713 \\
    - load & 295.000 & $\cdots$ & $\cdots$ & $\cdots$ & $\cdots$ & $\cdots$ & $\cdots$ & $\cdots$ & $\cdots$ \\
    - storage & 0.239 & 0.597 & 13.237 & $\cdots$ & 20.351 & 10.003 & $\cdots$ & 1.256 & 202.620 \\
    Supply (w/o storage) & 295.040 & 295.100 & 297.203 & $\cdots$ & 298.636 & 296.677 & $\cdots$ & 295.209 & 439.435 \\
    - backup & 0.028 & $\cdots$ & 0.075 & $\cdots$ & $\cdots$ & \textbf{9.431} & $\cdots$ & 0.010 & 0.068 \\
    - grid & 295.012 & 295.072 & 297.128 & $\cdots$ & 298.560 & 287.246 & $\cdots$ & 295.200 & 439.367 \\
    - storage & 0.199 & 0.498 & \textbf{11.034} & $\cdots$ & 16.964 & 8.338 & $\cdots$ & 1.047 & 168.897 \\
    \midrule
    \multicolumn{10}{c}{Yearly discharge cycles (\#)} \\
    \midrule
    Average (base case) & 0.000 & $\cdots$ & 150.000 & $\cdots$ & $\cdots$ & $\cdots$ & $\cdots$ & 0.000 & 145.295 \\
    Maximum (base case) & 0.000 & $\cdots$ & 150.000 & $\cdots$ & $\cdots$ & $\cdots$ & $\cdots$ & 0.000 & 150.000 \\
    Average (contingency) & 4.551 & 7.051 & 150.000 & $\cdots$ & $\cdots$ & $\cdots$ & $\cdots$ & 7.622 & 105.870 \\ 
    Maximum (contingency)  & 8.656 & $\cdots$ & \textbf{150.000} & $\cdots$ & $\cdots$ & $\cdots$ & $\cdots$ & 9.057 & 150.000 \\
    \midrule
    \multicolumn{10}{c}{Minimum scarcity supply ratio (-)} \\
    \midrule
    Backup (base case) & 0.000 & $\cdots$ & 1.000 & $\cdots$ & $\cdots$ & $\cdots$ & $\cdots$ & 0.000 & $\cdots$ \\
    Storage (base case) & 0.000 & $\cdots$ & 1.000 & $\cdots$ & $\cdots$ & $\cdots$ & $\cdots$ & 0.000 & 0.483 \\
    Backup (contingency) & 0.000 & $\cdots$ & 1.000 & $\cdots$ & $\cdots$ & $\cdots$ & $\cdots$ & 0.000 & $\cdots$ \\ 
    Storage (contingency) & 0.000 & $\cdots$ & \textbf{1.000} & $\cdots$ & \textbf{0.477} & 1.000 & $\cdots$ & 0.000 & 0.012 \\
\end{longtable}
}%
\input{pics/investment}
\newpage
\addcontentsline{toc}{section}{References (Appendix)}
\small
\linespread{1}\selectfont
\putbib[_mybib]
\end{bibunit}
\end{document}